\documentclass[reqno, a4paper]{amsart}
\usepackage[english]{babel}

\usepackage[vcentermath]{youngtab}
\usepackage{ytableau}

\usepackage{graphicx,mathrsfs,tikz,latexsym,ifthen,amsmath,amsfonts,amssymb,amsthm,stmaryrd,fancyhdr,amscd,amsbsy,amstext}
\usepackage{enumerate,enumitem,empheq,mathtools,mathpazo,verbatim,a4wide}
\mathtoolsset{showonlyrefs,showmanualtags}

\usepackage{epigraph}

\usepackage[utf8x]{inputenc}

\usepackage[toc,page]{appendix}
\usepackage[mathscr]{euscript} 
\allowdisplaybreaks

\usetikzlibrary{positioning,shapes,shadows,arrows}

\usepackage{color}
\definecolor{MyDarkBlue}{rgb}{0.15,0.25,0.45}

\usepackage{hyperref}
\hypersetup{anchorcolor=black, linktocpage=true,
colorlinks=true,
citecolor=red,
linkcolor=MyDarkBlue,
urlcolor=black,
pdfauthor={Francesco Sala and Olivier Schiffmann},
pdftitle={Cohomological Hall algebra of Higgs sheaves on a curve},
breaklinks=true,
plainpages=true
}

\usepackage[hyperpageref]{backref} 

\newcommand{\C}{\mathbb{C}}
\newcommand{\G}{\mathbb{G}}
\newcommand{\HH}{\mathbb{H}}

\newcommand{\R}{\mathbb{R}}
\newcommand{\N}{\mathbb{N}}
\newcommand{\PP}{\mathbb{P}}

\newcommand{\Q}{\mathbb{Q}}
\newcommand{\VV}{\mathbb{V}}
\newcommand{\Z}{\mathbb{Z}}

\newcommand{\Xscr}{\mathscr{X}}

\newcommand{\Uscr}{\mathscr{U}}
\newcommand{\Wscr}{\mathscr{W}}
\newcommand{\Zscr}{\mathscr{Z}}

\newcommand{\gfrak}{\mathfrak{g}}

\newcommand{\Ecal}{\mathcal{E}}
\newcommand{\Fcal}{\mathcal{F}}
\newcommand{\Gcal}{\mathcal{G}}
\newcommand{\Hcal}{\mathcal{H}}

\newcommand{\Lcal}{\mathcal{L}}

\newcommand{\Ocal}{\mathcal{O}}

\newcommand{\Tcal}{\mathcal{T}}

\newcommand{\Efrak}{\mathfrak{E}}
\newcommand{\Kfrak}{\mathfrak{K}}

\newcommand{\Hom}{\mathsf{Hom}}
\newcommand{\Ext}{\mathsf{Ext}}

\newcommand{\pt}{\mathsf{pt}}
\newcommand{\Frac}{\mathsf{Frac}}

\newcommand{\Ksf}{\mathsf{K}}
\newcommand{\Ksfnum}{\mathsf{K}^{\mathsf{num}}}
\newcommand{\Ksfnumplus}{\Ksf^{\mathsf{num},+}}
\newcommand{\Coh}{\mathsf{Coh}}
\newcommand{\CohL}{\Coh^{> \Lcal}}
\newcommand{\CohLp}{\Coh^{> \Lcal'}}
\newcommand{\Higgs}{\mathsf{Higgs}}
\newcommand{\Spec}{\mathsf{Spec}}
\newcommand{\Sym}{\mathsf{Sym}}

\newcommand{\bCoh}{\mathbf{Coh}}
\newcommand{\bCoha}{\bCoh_\alpha}
\newcommand{\bCohaL}{\bCoh_\alpha^{> \Lcal}}
\newcommand{\bCohaLp}{\bCoh_\alpha^{> \Lcal'}}
\newcommand{\bCohb}{\bCoh_\beta}
\newcommand{\bCohbL}{\bCoh_\beta^{> \Lcal}}

\newcommand{\bCohab}{\bCoh_{\alpha+\beta}}

\newcommand{\bCohabLp}{\bCoh_{\alpha+\beta}^{> \Lcal'}}
\newcommand{\bBun}{\mathbf{Bun}}
\newcommand{\bBuna}{\bBun_\alpha}

\newcommand{\bCohtildeab}{\widetilde{\bCoh_{\alpha, \beta}}}
\newcommand{\bCohtildeabLLp}{\widetilde{\bCoh_{\alpha, \beta}}^{> \Lcal, > \Lcal'}}
\newcommand{\bCohtildeabLoneLp}{\widetilde{\bCoh_{\alpha, \beta}}^{> \Lcal_{1}, > \Lcal'}}
\newcommand{\bCohtildeabLtwoLp}{\widetilde{\bCoh_{\alpha, \beta}}^{> \Lcal_{2}, > \Lcal'}}
\newcommand{\bCohtildeabLLone}{\widetilde{\bCoh_{\alpha, \beta}}^{> \Lcal, > \Lcal_1}}
\newcommand{\bCohtildeabLLtwo}{\widetilde{\bCoh_{\alpha, \beta}}^{> \Lcal, > \Lcal_2}}

\newcommand{\LCoha}{\mathbb{L}_{\bCoha}}
\newcommand{\TCoha}{\mathbb{T}_{\bCoha}}

\newcommand{\TCohaL}{\mathbb{T}_{\bCohaL}}

\newcommand{\bHiggs}{\mathbf{Higgs}}
\newcommand{\bHiggsa}{\bHiggs_\alpha}
\newcommand{\bHiggsb}{\bHiggs_\beta}

\newcommand{\bHiggsab}{\bHiggs_{\alpha+\beta}}

\newcommand{\bHiggsaL}{\bHiggs_\alpha^{> \Lcal}}
\newcommand{\bHiggsaLp}{\bHiggs_\alpha^{> \Lcal'}}

\newcommand{\bHiggsabLp}{\bHiggs_{\alpha+\beta}^{> \Lcal'}}
\newcommand{\bHiggsbL}{\bHiggs_\beta^{> \Lcal}}

\newcommand{\bHiggsaLone}{\bHiggs_\alpha^{> \Lcal_1}}
\newcommand{\bHiggsbLone}{\bHiggs_\beta^{> \Lcal_1}}

\newcommand{\bHiggsabLthree}{\bHiggs_{\alpha+\beta}^{> \Lcal_3}}

\newcommand{\bLam}{\mathbf{\Lambda}}
\newcommand{\bLama}{\mathbf{\Lambda}_{\alpha}}

\newcommand{\bHiggstildeab}{\widetilde{\bHiggs_{\alpha, \beta}}}
\newcommand{\bHiggstildeabLLp}{\widetilde{\bHiggs_{\alpha, \beta}}^{> \Lcal, > \Lcal'}}

\newcommand{\Quot}{\mathsf{Quot}}
\newcommand{\QaL}{\mathsf{Q}_\alpha^\Lcal}
\newcommand{\QaLp}{\mathsf{Q}_\alpha^{\Lcal'}}
\newcommand{\QaLLp}{\mathsf{Q}_\alpha^{\Lcal, \Lcal'}}
\newcommand{\QbL}{\mathsf{Q}_\beta^\Lcal}

\newcommand{\QabL}{\mathsf{Q}_{\alpha+\beta}^\Lcal}
\newcommand{\QabLLp}{\mathsf{Q}_{\alpha+\beta}^{\Lcal, \Lcal'}}

\newcommand{\QabLp}{\mathsf{Q}^{\Lcal'}_{\alpha+\beta}}

\newcommand{\QaLone}{\mathsf{Q}_\alpha^{\Lcal_1}}
\newcommand{\QbLone}{\mathsf{Q}_\beta^{\Lcal_1}}
\newcommand{\QaLoneLtwo}{\mathsf{Q}_\alpha^{\Lcal_1, \Lcal_2}}
\newcommand{\QbLoneLtwo}{\mathsf{Q}_\beta^{\Lcal_1, \Lcal_2}}

\newcommand{\QtildeabL}{\widetilde{\mathsf{Q}}^{\Lcal}_{\alpha, \beta}}
\newcommand{\QtildeabLLp}{\widetilde{\mathsf{Q}}^{\Lcal, \Lcal'}_{\alpha, \beta}}
\newcommand{\QtildeabLoneLp}{\widetilde{\mathsf{Q}}^{\Lcal_1, \Lcal'}_{\alpha, \beta}}
\newcommand{\QtildeabLtwoLp}{\widetilde{\mathsf{Q}}^{\Lcal_2, \Lcal'}_{\alpha, \beta}}

\newcommand{\QtildeabLoneLthree}{\widetilde{\mathsf{Q}}^{\Lcal_1, \Lcal_3}_{\alpha, \beta}}

\newcommand{\RtildeabLoneLtwoLp}{\widetilde{\mathsf{R}}^{(\Lcal_1,\Lcal_2),\Lcal'}_{\alpha,\beta}}

\newcommand{\RaLLp}{\mathsf{R}_\alpha^{\Lcal, \Lcal'}}
\newcommand{\RaLpL}{\mathsf{R}_\alpha^{\Lcal', \Lcal}}
\newcommand{\RabLLp}{\mathsf{R}_{\alpha+\beta}^{\Lcal, \Lcal'}}

\newcommand{\ual}{\underline{\alpha}}

\newcommand{\GL}{\mathsf{GL}}

\newcommand{\GaL}{\mathsf{G}_\alpha^\Lcal}
\newcommand{\GaLp}{\mathsf{G}_\alpha^{\Lcal'}}

\newcommand{\GbL}{\mathsf{G}^{\Lcal}_\beta}

\newcommand{\GabL}{\mathsf{G}_{\alpha+\beta}^\Lcal}
\newcommand{\GabLp}{\mathsf{G}_{\alpha+\beta}^{\Lcal'}}
\newcommand{\PabL}{\mathsf{P}_{\alpha,\beta}^{\Lcal}}
\newcommand{\PabLone}{\mathsf{P}_{\alpha,\beta}^{\Lcal_1}}
\newcommand{\PabLtwo}{\mathsf{P}_{\alpha,\beta}^{\Lcal_2}}

\newcommand{\glfrak}{\mathfrak{gl}}

\newcommand{\gaL}{\mathfrak{g}_\alpha^\Lcal}
\newcommand{\gaLp}{\mathfrak{g}_\alpha^{\Lcal'}}

\newcommand{\muaL}{\mu_\alpha^\Lcal}

\newcommand{\F}{\mathbb{F}}

\newcommand{\COHA}{\mathbf{AHA}}

\newcommand{\rk}{\operatorname{rk}}

\newcommand{\End}{\operatorname{End}}

\newcommand{\triend}{\parbox{2mm}{\hfill} \hfill\text{\hspace{0.2mm}}\hfill$\triangle$}
\newcommand{\ocend}{\parbox{2mm}{\hfill} \hfill\text{\hspace{0.2mm}}\hfill$\oslash$}

\newtheorem{theorem}{Theorem}
\newtheorem{proposition}[theorem]{Proposition}
\newtheorem{lemma}[theorem]{Lemma}
\newtheorem{corollary}[theorem]{Corollary}
\newtheorem{conjecture}[theorem]{Conjecture}
\newtheorem{corollary*}{Corollary}
\newtheorem*{theorem*}{Theorem}
\newtheorem*{proposition*}{Proposition}
\newtheorem*{conjecture*}{Conjecture}

\numberwithin{equation}{section}
\numberwithin{theorem}{section}

\theoremstyle{remark}
\newtheorem{ex}[equation]{Example}

\theoremstyle{remark}
\newtheorem{rem}[equation]{Remark}
\newenvironment{remark}{\begin{rem}}{\triend\end{rem}}

\theoremstyle{definition}
\newtheorem{defin}[equation]{Definition}
\newenvironment{definition}{\begin{defin}}{\ocend\end{defin}}

\title[Cohomological Hall algebra of Higgs sheaves on a curve]{Cohomological Hall algebra of Higgs sheaves on a curve}

\author[F.~Sala]{Francesco Sala}
\address[Francesco Sala]{Università di Pisa, Dipartimento di Matematica, Largo Bruno Pontecorvo 5, 56127 Pisa (PI), Italy}
\address{Kavli IPMU (WPI), UTIAS, The University of Tokyo, Kashiwa, Chiba 277-8583, Japan}
\curraddr{}
\email{\href{mailto:francesco.sala@unipi.it}{francesco.sala@unipi.it}}

\author[O.~Schiffmann]{Olivier Schiffmann}
\email{olivier.schiffmann@math.u-psud.fr}
\address[Olivier Schiffmann]{Laboratoire de Math\'ematiques, Universit\'e de Paris-Sud Paris-Saclay, B\^at. 425, 91405 Orsay Cedex, France}


\thanks{The first-named author is partially supported by World Premier International 
	Research Center Initiative (WPI), MEXT, Japan, by JSPS KAKENHI Grant number JP17H06598 and 
	by JSPS KAKENHI Grant number JP18K13402.} 

\begin{document}

\begin{flushright}
IPMU--18--0002
\end{flushright}

\vskip 1cm

\begin{abstract} We define the cohomological Hall algebra $\COHA_{\Higgs(X)}$ of the ($2$-dimensional) Calabi-Yau category of Higgs sheaves on a smooth projective curve $X$, as well as its nilpotent and semistable variants, in the context of an arbitrary oriented Borel-Moore homology theory. In the case of usual Borel-Moore homology, $\COHA_{\Higgs(X)}$ is a module over the (universal) cohomology ring $\HH$ of the stacks of coherent sheaves on $X$. We characterize its annihilator as a $\HH$-module, and we provide an explicit collection of generators (the collection of fundamental classes $[\bCoh_{r,\, d}]$ of the zero-sections of the maps $\bHiggs_{r,\, d} \to \bCoh_{r,\, d}$, for $r > 0, d \in \Z$ and $r=0$ and $d\in \Z_{>0}$).
\end{abstract}

\maketitle

\epigraph{\emph{salimmo sù, el primo e io secondo,\newline
tanto ch’i’ vidi de le cose belle\newline
che porta ’l ciel, per un pertugio tondo.}}{Dante Alighieri, \emph{la Divina Commedia, Inferno, Canto XXXIV}}
 
\bigskip\tableofcontents

\bigskip\section{Introduction}

The aim of this paper is to define and begin the study of \textit{cohomological Hall algebras} in the context of moduli stacks of Higgs bundles on smooth projective curves over a field $k$.

Let us recall that a Higgs bundle on a complex Riemann surface $X$ of arbitrary genus is a pair $(\Fcal, \phi\colon \Fcal \to \Fcal\otimes \omega_X)$ consisting of a vector bundle $\Fcal$ and a morphism $\phi$ called a \emph{Higgs field}; here $\omega_X$ is the canonical line bundle of $X$. Moduli spaces of \emph{stable} Higgs bundles of fixed rank and degree over $X$ were introduced by Hitchin in the late 80's \cite{art:hitchin1987, art:hitchin1987-II} --- see, e.g., \cite[Appendix]{book:wells2008} and \cite{art:casalainawise2017}, respectively, for a differential-geometric point of view and for an algebro-geometric one to the Hitchin moduli spaces. These moduli spaces have a rich geometry: for example, they are smooth quasi-projective varieties and, from a differential point of view, they are endowed with a complete hyperk\"ahler metric. In addition, the map which associates with any stable Higgs bundle $(\Fcal, \phi)$ the characteristic polynomial of $\phi$ defines a complete integrable system, called the \emph{Hitchin fibration}. The preimage with respect to zero of the Hitchin fibration is the so-called \emph{global nilpotent cone}, which parametrizes stable Higgs bundles $(\Fcal, \phi)$ with nilpotent Higgs field $\phi$. Since its introduction, the Hitchin moduli space has played a preeminent role in the theory of moduli spaces, integrable systems, mirror symmetry, number theory, and string and gauge theories. 

In the well-documented analogy between (smooth, projective) curves and quivers the role of the Hitchin moduli stack is played by the \emph{preprojective stack}, and the analog of the global nilpotent cone is the \emph{Lusztig nilpotent stack}. As for the moduli space of (stable) Higgs bundles, the closest analog is another family of non-compact hyperk\"ahler manifolds which share similar geometric properties, the \emph{Nakajima quiver varieties} --- introduced in \cite{art:nakajima:1994-3}. They admit as well a canonical projective morphism to an affine variety, which is the affinization map --- such a morphism plays the role of the Hitchin fibration; the global nilpotent cone is replaced by the \emph{Lagrangian Nakajima quiver variety}. See \cite{art:schiffmann2006-II, art:ginzburg2012} for an introduction to the theory of quiver varieties, and \cite{art:bozec2016} for more details on the case of quivers with edge loops.

As illustrated by the classical results of Nakajima and others, the (co)homology (or K-theory) of quiver varieties is extremely rich from the point of view of representation theory, and many of its topological invariants have representation-theoretic meanings. For instance, the computation of the Poincar\'e polynomials of Nakajima quiver varieties associated with an arbitrary quiver was done by Hausel in \cite{art:hausel2010}, where he showed that such a polynomial is related to the \emph{Kac's A-polynomial} of the quiver. Recall that, as proved by Kac and Stanley \cite{art:kac1982, art:kac1983} the number of geometrically indecomposable $\F_q$-representations of a quiver $Q$ of given dimension $\mathbf{d}$ is given by a polynomial $A_{Q, \mathbf{d}}(q)$ in $q$, called Kac's $A$-polynomial. Another geometric interpretation of the Kac's A-polynomial is the one in terms of the Poincar\'e polynomial of the preprojective stack, i.e., the stack of representations of the preprojective algebra $\Pi_Q$ associated with $Q$ \footnote{One has also a nilpotent version of such a relation, by considering from the algebraic side nilpotent versions of the Kac's A-polynomial and from the geometric side the generalizations of Lusztig's nilpotent variety introduced in \cite{art:bozecschiffmannvasserot2017,art:schiffmannvasserot2017} (see also \cite{art:bozec2015, art:bozec2016}).}. Much more recently, this relation between a polynomial of geometric nature, such as the Poincar\'e polynomials of Nakajima quiver variety associated with $Q$ and of the stack of representations of $\Pi_Q$, and a polynomial of representation-theoretic nature, such as the Kac's A-polynomial $A_{Q, \mathbf{d}}(q)$, has been ``categorified" in the following way (cf.\ \cite{art:schiffmannvasserot2017}): there exists an associative algebra structure on the Borel-Moore homology of the stack of representations of $\Pi_Q$ --- the so-called \emph{cohomological Hall algebra} --- whose Hilbert series is given by the Kac's A-polynomial of $Q$. Moreover, such an algebra is conjecturally\footnote{The conjecture is true for finite and affine quivers. At the moment there is only a partial result for general quivers: see \cite{art:schiffmannvasserot2017-II}.} isomorphic to the positive part of the Yangian algebra $\mathsf{Y}(\gfrak_Q)$ of the Maulik-Okounkov graded Lie algebra $\gfrak_Q$ (cf.\  \cite{art:maulikokounkov2012, art:schiffmannvasserot2017-II} for a definition of $\gfrak_Q$). What's more, this algebra acts on the Borel-Moore homology of Nakajima quiver varieties associated with the same quiver, and such an action extends to a larger\footnote{At least when the quiver is not of finite type.} algebra of symmetries Nakajima's construction of representations of $\mathsf{U}(\gfrak_Q)$ on the Borel-Moore homology of Nakajima quiver varieties. 

Let us return to the curve case, for which the situation is (from that point of view) much less developed. The Poincar\'e polynomial of the moduli stack of Higgs bundles is ill-defined (i.e. the Betti numbers are in general infinite), and the Betti numbers of the moduli spaces of stable Higgs bundles on curve $X$, for coprime rank and degree, were only recently computed in terms of the \textit{Kac polynomial} of $X$ in \cite{art:schiffmann2016, art:mellit2017}. We refer to \textit{loc. cit.} for a precise definition of these Kac polynomials $A_{r,\, g}(z_1, \ldots, z_{2g})$, which depend on the rank $r$ and the genus $g$ of the curve, and whose evaluation at the Weil numbers $(\sigma_1, \ldots, \sigma_{2g})$ of the curve is equal to the number of geometrically indecomposable vector bundles of rank $r$, degree $d$ on the curve $X$ defined over $\mathbb{F}_q$.

It is natural to wonder if in the curve case also there is a deeper representation theoretic result behind such an enumerative relation. The aim of the present paper is to perform the first step of this program, namely to construct the \emph{cohomological Hall algebra} attached to the stacks of Higgs sheaves over a smooth projective curve $X$ of genus $g$. Note that we consider here the entire stack $\bHiggs(X)\coloneqq\bigsqcup_{r, \, d} \bHiggs_{r, \, d} $ and not only of its stable part. Although our main potential applications in mind are in the context of Borel-Moore homology or K-theory, we develop the theory of these cohomological Hall algebra for an arbitrary free oriented Borel-Moore (OBM) homology theory (as is done in \cite{art:yangzhao2014} in the context of quivers).
 
The construction and detailed study of the cohomological Hall algebra for the stack of Higgs \emph{torsion} sheaves is the subject of the recent work by Minets in \cite{art:minets2018}. Our first main result extends the construction to the higher rank case.
\begin{theorem}[Theorem \ref{theorem:defproduct}] 
Let $X$ be an irreducible smooth projective curve over a field $k$. Let $A$ be either the Borel-Moore homology or an arbitrary free oriented Borel-Moore homology theory\footnote{Since we are dealing with algebraic stacks with infinitely many irreducible components, we consider rather a subgroup $A^0 \subseteq A$ of classes satisfying some support condition, see Section \ref{sec:BM}.}. Then there is a canonical graded associative algebra structure on 
\begin{align}
\COHA_{\Higgs(X)}\coloneqq\bigoplus_{r,\, d} A_\ast(\bHiggs_{r,\, d}) \ .
\end{align}
\end{theorem}

There are some natural variants of this algebra, in which we replace the stacks $\bHiggs_{r, \, d}$ by the global nilpotent cones $\bLam_{r, \, d}$ or the stacks of semistable Higgs bundles $\bHiggs_{r, \, d}^{\mathsf{ss},\, \nu}$, of a fixed slope $\nu$.  We can as well introduce an equivariant parameter coming from the action of $T\coloneqq\mathbb{G}_m$ by dilations on the Higgs field.
\begin{corollary}[Corollaries~\ref{cor:Higgsvariantnilp}, \ref{cor:Higgsvariantss}] 
There are canonical graded associative algebra structures on 
\begin{align}
\COHA_{\bLam}\coloneqq\bigoplus_{r, \, d} A_\ast(\bLam_{r, \, d})\ , \qquad \COHA_{\Higgs^{\mathsf{ss},\, \nu}(X)}\coloneqq\bigoplus_{\substack{d/r=\nu}} A_\ast(\bHiggs^{\mathsf{ss}}_{r, \, d}) \quad \text{for all } \nu \in \mathbb{Q} \cup \{\infty\}
\end{align}
and on their $T$-equivariant cousins $\COHA^T_{\bLam}$, $\COHA^T_{\Higgs^{\mathsf{ss},\, \nu}(X)}$. 
\end{corollary}

There are some strong relations between these variants and the original cohomological Hall algebra of $\bHiggs(X)$. For instance, the proper pushforward map $\COHA_{\bLam} \to \COHA_{\Higgs(X)}$ is an algebra homomorphism. Likewise, the open restriction map $\COHA_{\Higgs^\nu(X)} \to \COHA_{\Higgs^{\mathsf{ss},\, \nu}(X)}$ is an algebra homomorphism, where
\begin{align}
\COHA_{\Higgs^{\nu}(X)}\coloneqq\bigoplus_{d/r=\nu} A_\ast(\bHiggs_{r, \, d})\ .
\end{align}
Moreover, the proper pushforward induces an isomorphism of localized algebras
\begin{align}\label{eqintro:1}
\COHA^T_{\bLam} \otimes_{A_T(\pt)} \Frac(A_T(\pt)) \stackrel{\sim}{\longrightarrow} \COHA^T_{\Higgs(X)} \otimes_{A_T(\pt)} \Frac(A_T(\pt)).
\end{align}
(see Proposition~\ref{prop:localizationT}).

Although the definition of cohomological Hall algebras can be given for an arbitrary free OBM theory in a very uniform fashion, the properties of these algebras strongly depend on the choice of the OBM theory. Our results concerning the structure of $\COHA_{\Higgs(X)}$ are for the moment restricted to the cases of usual Borel-Moore homology (or Chow groups). So \emph{we assume until the end of this introduction that $A=H_\ast$ and we restrict ourselves to $A^0$}. The cohomology ring of the stack $\bCoh(X)$ of coherent sheaves on $X$ acts on $\COHA_{\Higgs(X)}$ by pullback to $\bHiggs(X)$ and cap product. By Heinloth's generalization of the Atiyah-Bott theorem (see Theorem \ref{T:Heinloth}), this ring is (freely\footnote{Only in the positive rank case.}) generated by tautological classes, and we can define a universal ring (in fact, (co)commutative Hopf algebra) $\HH$ which acts on $\COHA_{\Higgs(X)}$ (and on all its cousins).

The second main result of the paper concerns torsion-freeness. It can be seen as a key technical step to embed our algebra in a bigger shuffle-type algebra (as done in the rank zero case in \cite[Section~3]{art:minets2018}).
\begin{theorem}[Theorem \ref{T:torsionfree}] 
Let $\alpha\in (\Z^2)^+$. Then $H^T_\ast(\bLama)$ is a torsion-free $H^\ast(\bCoha)\otimes \Q[t]$-module.
\end{theorem}
From the analogy with the case of quivers, it is natural to expect that in fact $\COHA^T_{\Lambda}$ is of generic rank one (not free!), but we do not prove this here.

Our final main result, in a spirit similar to \cite[Theorem~B (e)]{art:schiffmannvasserot2017}, provides a family of generators for $\COHA_{\Lambda}$.
\begin{theorem}[Theorem \ref{T:gen}, Corollary \ref{cor:gen2}] 
The $\HH$-algebra $\COHA_{\Lambda}$ is generated by the collection of fundamental classes $\{[\bLam_{(r,\, d)}]\}_{r, \, d}$ of the zero sections of the projections $\bHiggs_{r, \, d} \to \bCoh_{r, \, d}$. 
\end{theorem}
Of course, using Formula \eqref{eqintro:1} we may deduce similar results for $\COHA_{\Higgs(X)}$.

Let us conclude this introduction with some heuristics and speculations concerning the structure and representation theory of $\COHA_{\Higgs(X)}$. 

First, we expect that our cohomological Hall algebra acts on the (oriented) Borel-Moore homology of moduli spaces of stable Higgs bundles (of fixed slope) and of Minets' generalization of Nakajima quiver varieties\footnote{These moduli spaces parametrize stable point in the cotangent stack of the stack of \emph{coframed pairs}, i.e., pairs $(\Fcal, \Ecal\to \Fcal)$ where $\Ecal, \Fcal$ are coherent sheaves on $X$, and $\Ecal$ is fixed, see \cite{phdthesis:minets2018-II}.}. Slightly more generally, it is natural to expect that it will also act on suitable moduli spaces of stable and framed sheaves on smooth (stacky) surfaces containing the curve $X$ as an embedded divisor of self-intersection $2(g-1)$\footnote{More general embedded curves would require constructing a cohomological Hall algebra for \emph{arbitrarily} twisted Higgs bundles.}. If $X=\PP^1$, examples of such surfaces are those described in \cite[Section~2, Remarks (ii)]{art:ginzburgkapranovvasserot1995} and the stacky surfaces defined in \cite{art:bruzzopedrinisalaszabo2016}.

Next, by analogy with the case of quivers, one can hope for a strong relation between the algebra $\COHA_{\Higgs(X)}$ and the (usual) Hall algebra of curves of genus $g$ over finite fields. Very slightly more precisely, one would expect the existence of a graded Lie algebra $\mathfrak{g}_{g}$ whose Hilbert series is given by the Kac polynomials $A_{r,\, g}$, which would be a 'generic' form of the Hall-Lie algebra of curves of genus $g$ on the one hand (cf.\ \cite[Section~8.3]{art:schiffmann2016} for the definition of such a Lie algebra), and whose affinization (or Yangian) would be isomorphic to $\COHA_{\Higgs(X)}$.

Finally, by the nonabelian Hodge correspondence \cite{art:simpson1994-II}, moduli spaces of stable rank $r$ Higgs bundles on $X$ are diffeomorphic to (twisted) character varieties of $X$ for the group $\mathsf{GL}(r)$. The topology of the latter moduli space has been extensively studied by Hausel, Letellier and Rodriguez-Villegas (cf.\ \cite{art:hauselletellierrodriguezvillegas2013} and the conjectures stated therein). In \cite{art:portasala2019}, the first-named author and Mauro Porta have constructed cohomological Hall algebras for the moduli stack of vector bundles with flat connections on $X$ and for the character stack of $X$ for the group $\mathsf{GL}(r)$, respectively. In addition, in \emph{loc.cit.}, some relations between the three different cohomological Hall algebras have been established.  We propose the following conjecture, which can be seen as a representation theoretical version of the nonabelian Hodge correspondence.\footnote{There is an equivalent construction of cohomological Hall algebra of the untwisted character stack for $\mathsf{GL}(r)$ by means of Kontsevich-Soibelman critical CoHA (see \cite{art:davisonmeinhardt2016, Davison_character}). The conjecture \ref{conj} was also made by B. Davison (see e.g. \cite{art:davison2016}).}
\begin{conjecture}\label{conj}
The algebra $\COHA_{\Higgs^{\mathsf{ss},0}(X)}$ is isomorphic to the cohomological Hall algebra of the genus $g$ untwisted character stack.
\end{conjecture}

This paper is organized as follows. Section \ref{sec:CohHiggs} provides notations and serves as a reminder concerning stacks of coherent and Higgs sheaves on smooth projective curves. The cohomological Hall algebras (or $A$-homological Hall algebras) are defined in Section \ref{sec:COHAHiggs}. From this point on, we restrict ourselves to the context of Borel-Moore homology. In Section \ref{sec:torsionfreeness} we introduce the universal cohomology ring of the stacks of coherent sheaves on curves of a fixed genus, and prove the torsion-freeness result. Section \ref{sec:generation} is devoted to the generation theorem.

\subsection*{Acknowledgements}

The first seed of the present paper can be traced back to some discussions during two visits of the first-named author to Paris: the first one was under the umbrella of the Research in Paris program supported by the Institut Henri Poincaré, while the host of the second visit was the Université of Paris-Sud. The first-named author thanks both institutions for the hospitality and support. In addition, some of the results of the present paper were presented during the workshop on ``Hitchin systems in Mathematics and Physics" (February 2017, Perimeter Institute, Canada) and the workshop on ``Geometric Representation Theory" (July 2017, University of Glasgow, UK). The first-named author thanks the organizers and the participants of both workshops for interesting discussions. Finally, we thank B. Davison, D. E. Diaconescu, S. Meinhardt, A. Minets, A. Negut, Y. Soibelman and É. Vasserot for interesting discussions and comments.

\bigskip\section{Stacks of coherent and Higgs sheaves on a curve}\label{sec:CohHiggs}

In this section we introduce the stacks of coherent and Higgs sheaves on smooth projective curves, and recall some of their key properties. Because our construction of the multiplication in the cohomological Hall algebras uses the local charts defined in terms of Quot schemes, we go into some depth in describing the latter explicitly.

\subsection{The curve}

Let $X$ be an irreducible smooth projective curve of genus $g$ over a field $k$, and $\omega_X$ its canonical line bundle. As usual, we denote by $\rk(\Fcal), \deg(\Fcal)$ the rank and degree of a coherent sheaf $\Fcal$ on $X$ and by 
\begin{align}
\mu(\Fcal) =\frac{\deg(\Fcal)}{\rk(\Fcal)}\in \mathbb{Q} \cup \{ \infty\}
\end{align} 
its slope. Denote by $\Coh(X)$ the category of coherent sheaves on $X$. It is an abelian category of homological dimension one. Denote by $\Ksf(X)$ the \emph{Grothendieck group} of $X$ and by $\left[\Fcal\right]$ the class of a coherent sheaf $\Fcal$. Let $\Ksf(X)^+$ be the semigroup of $\Ksf(X)$ consisting of classes of the form $\left[\Fcal\right]$, for a coherent sheaf $\Fcal$ on $X$. There are natural maps 
\begin{align}
\rk\colon \Ksf(X)\to \Z_{\geq 0}\qquad \text{and} \qquad\deg\colon \Ksf(X)\to \Z
\end{align}
assigning to $[\Fcal]$ the rank and degree of $\Fcal$ respectively. This yields a projection $\Ksf(X)\to \Ksfnum(X)$, where $\Ksfnum(X)\coloneqq\Z^2$ is the \emph{numerical Grothendieck group} of $X$. We define the (numerical) class of a coherent sheaf $\Fcal$ as the pair $\overline{\Fcal}\coloneqq\big(\rk(\Fcal), \deg(\Fcal)\big)$. We accordingly set
\begin{align}
\Ksfnumplus(X)=\{(r,d)\in \Z^2\;\vert \; r >0, d \in \Z\;\text{or}\; r=0, d \geq 0\}\eqqcolon (\Z^2)^+\ .
\end{align}
Finally, recall that the Euler form on $\Ksf(X)$, which descends to $\Ksfnum(X)$, is explicitly given by the following formula:
\begin{multline}
\langle\, \overline{\Ecal}, \overline{\Fcal}\,\rangle\coloneqq\dim \Hom(\Ecal,\Fcal)-\dim \Ext^1(\Ecal, \Fcal)\\
=(1-g)\rk(\Ecal)\rk(\Fcal)+(\rk(\Ecal)\deg(\Fcal)-\rk(\Fcal)\deg(\Ecal))\ .
\end{multline}

\subsection{Stack of coherent sheaves}\label{sec:stackcoherentsheaves}

For $\alpha\in \Ksfnumplus(X)$, let $\bCoha$ be the stack parameterizing coherent sheaves on $X$ of class $\alpha$. It is a smooth algebraic stack, locally of finite type over $\Spec(k)$, and irreducible of dimension $-\langle \alpha, \alpha \rangle$; in addition, $\bCoha$ is equipped with a \emph{tautological sheaf} $\Efrak_\alpha\in \Coh\big(\bCoha\times X)$ (see \cite[Théorème~4.6.2.1]{book:laumonmoretbailly2000}; the smoothness follows, e.g., from the description of an atlas of $\bCoha$ given below). Since $\bCoha$ is smooth, the cotangent complex\footnote{The theory of cotangent complexes for algebraic stacks is developed in \cite[Chapter~16]{book:laumonmoretbailly2000} and \cite{art:olsson2007}.} $\LCoha$ of $\bCoha$ is perfect (hence dualizable) of Tor-amplitude [0, 1] (cf.\ \cite[Proposition~17.10]{book:laumonmoretbailly2000}); the dual complex, the tangent complex $\TCoha$, can be described explicitly as (cf.\ \cite[Section~2]{art:portasala2019})
\begin{align}\label{eq:tangentcomplex}
\TCoha =\R p_\ast\, \R \Hcal om(\Efrak_\alpha, \Efrak_\alpha)[1]\ , 
\end{align}
where $p\colon \bCoha\times X\to \bCoha$ is the projection.
  
For later purposes, let us give an atlas for $\bCoha$; this will be used in Section~\ref{sec:COHAHiggs} for the definition of the cohomological Hall algebra associated with the moduli stacks of Higgs sheaves. Let us fix a line bundle $\Lcal$ on $X$. We will say that a coherent sheaf $\Fcal$ is \emph{strongly generated by} $\Lcal$ if the canonical morphism $\Hom(\Lcal,\Fcal) \otimes \Lcal \to \Fcal$ is surjective and $\Ext^1(\Lcal,\Fcal)=\{0\}$. We denote by $\CohL(X)\subset \Coh(X)$ the full subcategory of coherent sheaves on $X$ which are strongly generated by $\Lcal$. Note that $\CohL(X)$ is stable under quotients and extensions and that
 \begin{align}
 \dim \Hom(\Lcal,\Fcal)=\langle \Lcal,\Fcal\rangle
 \end{align}
for all $\Fcal \in \CohL(X)$. Let $u_\alpha^\Lcal\colon \bCohaL\hookrightarrow \bCoha$ be the open substack of $\bCoha$ parameterizing sheaves strongly generated by $\Lcal$ and of class $\alpha$. We call $\bCohaL$ a \emph{local chart} of $\bCoha$. The stack $\bCohaL$ can be realized as a global quotient stack as follows. Let $\Quot_\alpha^\Lcal\coloneqq\Quot_{X/k}\big(\Lcal\otimes k^{\, \langle \overline{\Lcal}, \alpha\rangle}, \alpha\big)$ be the Quot scheme parameterizing isomorphism classes of quotients $\phi \colon \Lcal\otimes k^{\, \langle \overline{\Lcal}, \alpha\rangle}\twoheadrightarrow \Fcal$ such that $\overline{\Fcal}=\alpha$ (see \cite[Section~2.2]{book:huybrecthslehn2010} for an introduction to the theory of Quot schemes). This is a projective $k$-scheme, which is singular in general, of finite type and carries a canonical $\GaL\coloneqq\GL(k,\langle \overline{\Lcal},\alpha\rangle)$-action defined by $g \cdot \phi\coloneqq\phi \circ (\mathsf{id}_\Lcal \otimes g^{-1})$.  Its Zariski tangent space at a point $\big[\phi \colon \Lcal\otimes k^{\, \langle \overline{\Lcal},\alpha\rangle}\twoheadrightarrow \Fcal\big]$ is $\Hom(\ker(\phi), \Fcal)$, while the obstruction to the smoothness lies in $\Ext^1(\ker(\phi),\Fcal)$. Consider the open subscheme $\QaL \subset \Quot_\alpha^\Lcal$ be the open subscheme whose $k$-points are
\begin{align}
 \QaL(k)\coloneqq\{\big[\phi\colon\Lcal\otimes k^{\, \langle \overline{\Lcal},\alpha\rangle}\twoheadrightarrow \Fcal\big] \in \Quot_\alpha^\Lcal(k)\;\vert\; \phi_\ast\colon k^{\,\langle \overline{\Lcal},\alpha\rangle} \stackrel{\sim}{\longrightarrow} \Hom(\Lcal,\Fcal)\}\ .
\end{align}
\begin{proposition}
The following hold:
\begin{itemize}
\item[(i)] The scheme $\QaL$ is $\GaL$-invariant and there is a canonical isomorphism of algebraic stacks
\begin{align}
\bCohaL \simeq \big[ \QaL/\GaL\big]\ .
\end{align}
\item[(ii)] $\QaL$ is smooth and reduced.
\end{itemize} 
\end{proposition}
Statement (i) is shown for example in the proof of \cite[Théorème~4.6.2.1]{book:laumonmoretbailly2000}, while (ii) follows from the vanishing of $\Ext^1(\ker(\phi),\Fcal)$ for any point $\big[\phi \colon \Lcal\otimes k^{\, \langle \overline{\Lcal},\alpha\rangle}\twoheadrightarrow \Fcal\big]\in \QaL$ and \cite[Theorem~5.3]{book:newstead1978}. We may think of $\QaL$ as the fine moduli space parameterizing pairs $(\Fcal,u)$ where $\Fcal \in \CohL(X)$ is of class $\alpha$ and $u$ is a trivialization $k^{\, \langle \overline{\Lcal},\alpha\rangle} \stackrel{\sim}{\longrightarrow} \Hom(\Lcal,\Fcal)$. 

Now, we shall provide an explicit description of $(u_\alpha^\Lcal)^\ast \TCoha$ and $(u_\alpha^\Lcal)^\ast \LCoha$, which will be useful later on. On $\bCohaL \simeq \big[ \QaL/\GaL\big]$, the tautological sheaf $\Efrak_\alpha$ fits into a short exact sequence of tautological $\GaL$-equivariant sheaves on $\QaL \times X$
\begin{align}
0\to \Kfrak_\alpha^\Lcal \to \Ocal_{\QaL}^{\oplus\,  \langle \overline{\Lcal},\alpha\rangle}\boxtimes \Lcal \to \Efrak_\alpha^\Lcal=(u_\alpha^\Lcal)^\ast\Efrak_\alpha\to 0
\end{align}
such that the fibers over $\big\{\big[\phi\colon \Lcal\otimes k^{\, \langle \overline{\Lcal},\alpha\rangle}\twoheadrightarrow \Fcal\big]\big\}\times X$ are
\begin{align}
\Kfrak_\alpha^\Lcal\vert_{\{[\phi]\}\times X}=\ker \phi \quad\text{and}\quad \Efrak_\alpha^\Lcal \vert_{\{[\phi]\}\times X}=\Fcal=\Lcal\otimes k^{\, \langle \overline{\Lcal},\alpha\rangle}/\ker \phi\ .
\end{align}

Let $p_\alpha^\Lcal\colon \QaL \times X \to \QaL$ denote the projection. Since $\Ext^1(\Kfrak_\alpha^\Lcal\vert_{\{[\phi]\}\times X}, \Efrak_\alpha^\Lcal\vert_{\{[\phi]\}\times X})=\{0\}$ for any $[\phi]\in \QaL$, the complex $\mathbb{R}(p_\alpha^\Lcal)_\ast((\Kfrak_\alpha^\Lcal)^\vee \otimes \Efrak_\alpha^\Lcal)$ is a locally free sheaf on $\QaL$ of rank $\dim \QaL$, which we will simply denote by $\Hom(\Kfrak_\alpha^\Lcal,\Efrak_\alpha^\Lcal)$. Such a locally free sheaf coincides with the tangent bundle $\Tcal_{\QaL}$ of $\QaL$. 

Likewise, the fiber of $\mathbb{R}(p_\alpha^\Lcal)_\ast((\Ocal_{\QaL}^{\oplus\,  \langle \overline{\Lcal},\alpha\rangle}\boxtimes \Lcal^\vee) \otimes \Efrak_\alpha^\Lcal)$ over $\big\{\big[\phi\colon \Lcal\otimes k^{\, \langle \overline{\Lcal},\alpha\rangle}\twoheadrightarrow \Fcal\big]\big\}\times X$ is identified, via $\phi_\alpha$, with  $\gaL\coloneqq \glfrak(k,\langle \overline{\Lcal},\alpha\rangle)$, the Lie algebra of $\GaL$, hence $\mathbb{R}(p_\alpha^\Lcal)_\ast((\Ocal_{\QaL}^{\oplus\,  \langle \overline{\Lcal},\alpha\rangle}\boxtimes \Lcal^\vee) \otimes \Efrak_\alpha^\Lcal)\simeq \gaL \otimes \mathcal{O}_{\QaL}$. Collecting the above, from Formula \eqref{eq:tangentcomplex} we get that
\begin{align}
(u_\alpha^\Lcal)^\ast \TCoha\simeq \big[\gaL \otimes \Ocal_{\QaL} \stackrel{\delta_\alpha^\Lcal}{\longrightarrow}\Tcal_{\QaL}\big]\ ,
\end{align}
where the complex on the right-hand-side is concentrated in degree [-1, 0]. Thus, 
\begin{align}\label{eq:cotancomplex}
(u_\alpha^\Lcal)^\ast \LCoha\simeq \big[\Tcal_{\QaL}^\ast  \stackrel{\tilde \mu_\alpha^\Lcal}{\longrightarrow} (\gaL)^\ast \otimes \Ocal_{\QaL}\big]\ ,
\end{align}
where $\tilde \mu_\alpha^\Lcal$ (the moment map) is obtained by dualizing the canonical restriction morphism $\delta_\alpha^\Lcal$ defined, at the level of points $\big[\phi \colon \Lcal\otimes k^{\, \langle \overline{\Lcal},\alpha\rangle}\twoheadrightarrow \Fcal\big]$, as $\delta_\alpha^\Lcal(u)=( \phi \circ u)\vert_{\ker \phi}$ for $u \in \gaL = \End(\Lcal \otimes k^{\langle \Lcal, \alpha\rangle})$.

Next, let us realize $\bCoha$ as an \emph{ind-algebraic stack}\footnote{We consider ind-algebraic stacks in a very broad sense, as stated in \cite[Definition~4.2.1]{art:emertongee2015}.}. As a first thing, let us make explicit the inductive system of $\bCohaL$'s. Let $\mathfrak{Pic}(X)$ be the groupoid formed by all line bundles on $X$ with their isomorphisms. We define the following preorder $\prec$ on (the set of objects of) $\mathfrak{Pic}(X)$ such that it is endowed with the structure of a directed groupoid: we say that $\Lcal \prec \Lcal'$, for two line bundles $\Lcal$ and $\Lcal'$, if $\Lcal'$ is strongly generated by $\Lcal$. In that situation, any coherent sheaf $\Fcal$, which is strongly generated by $\Lcal'$, is also strongly generated by $\Lcal$. Hence, we have a chain of open embeddings
\begin{align}
\bCohaLp \subseteq \bCohaL\subseteq  \bCoha
\end{align}
coming from the inclusions of full subcategories $\CohLp(X) \subset \CohL(X)$. We will describe these embeddings explicitly in local atlases, namely we will
construct the map $j_{\Lcal,\Lcal',\alpha}\colon \big[ \QaLp/\GaLp\big] \hookrightarrow \big[\QaL/\GaL\big]$. 

To define $j_{\Lcal,\Lcal',\alpha}$, we shall provide another equivalent description of $\bCohaLp$ as a global quotient stack. For this, consider the open subscheme $\QaLLp \subset \QaL$ consisting of all points $\big[\phi\colon \Lcal \otimes k^{\, \langle \overline{\Lcal}, \alpha\rangle} \twoheadrightarrow \Fcal\big]$ for which $\Fcal \in \CohLp(X)$. Then $\bCohaLp\simeq  \big[ \QaLLp/\GaL\big]$ and we have a canonical open embedding $\big[ \QaLLp/\GaL\big]\hookrightarrow \big[\QaL/\GaL\big]=\bCohaL$. Now, we need to compare the two realizations $\big[\QaLLp/\GaL\big]$ and $ \big[ \QaLp/\GaLp\big]$ by providing a canonical explicit isomorphism between them. Let $p_X\colon \QaLp\times X\to X$ be the projection. Consider the $\GaLp$-equivariant sheaf $\Hom(p_X^\ast\Lcal, \Efrak_\alpha^{\Lcal'})$ over $\QaLp$, which by the same reasoning as above is locally free and of rank $\langle \Lcal, \alpha\rangle$. Let $\RaLLp$ be the total space of the associated $\GaL$-bundle. Therefore $\RaLLp$ carries an action of $\GaL \times \GaLp$ and 
\begin{align}
\RaLLp /  \GaL  \simeq \QaLp, \qquad 
\big[ \RaLLp /  \GaL \times \GaLp\big] \simeq \bCohaLp \ ,
\end{align}
the first isomorphism being in the category of $\GaLp$-schemes\footnote{Here and in the following, we call \emph{$G$-scheme} a scheme endowed with an action of an algebraic group $G$.}. Likewise, let $\RaLpL$ be the total space of the $\GaL$-equivariant $\GaLp$-bundle
$\Hom(p_X^\ast \Lcal', \Efrak_\alpha^\Lcal)$ over $\QaLLp$, so that
\begin{align}
 \RaLpL /  \GaLp  \simeq \QaLLp, \qquad 
\big[ \RaLpL /  \GaL \times \GaLp\big] \simeq \bCohaLp\ ,
\end{align}
the first isomorphism being in the category of $\GaL$-schemes.

\begin{lemma}\label{L:rall} 
There is a canonical isomorphism $\RaLLp \simeq \RaLpL$ in the category of $\GaL \times \GaLp$-schemes.
\end{lemma}
\proof
By construction, $\RaLLp$ represents the contravariant functor $\mathsf{Aff}/k \to (\mathsf{Sets})$ which assigns to an affine $k$-variety $S$ the set of pairs $([\phi],u)$ where $\big[\phi\colon \Lcal' \boxtimes \Ocal_S^{\oplus\,  \langle \overline{\Lcal'}, \alpha\rangle} \twoheadrightarrow \Fcal\big]$ belongs to $\QaLp(S)$ and $u$ is a trivialization $\mathcal{O}_S^{\oplus \langle \Lcal, \alpha\rangle} \simeq \Hom(\Lcal \boxtimes \Ocal_S, \Fcal)$. Similarly, $\RaLpL$ represents the contravariant functor $\mathsf{Aff}/k \to (\mathsf{Sets})$ which assigns to an affine $k$-variety $S$ the set of pairs $([\psi],v)$ where
$\big[\psi\colon \Lcal \boxtimes \mathcal{O}_S^{\oplus\, \langle \overline{\Lcal}, \alpha\rangle} \twoheadrightarrow \Fcal\big]$ belongs to $\QaLLp(S)$ and $v$ is a trivialization $\Ocal_S^{\oplus \, \langle \overline{\Lcal'}, \alpha\rangle} \simeq \Hom(\Lcal' \boxtimes \mathcal{O}_S, \Fcal)$. The isomorphism between the two functors is given by the assignment $(\phi,u) \mapsto (\overline{u},\overline{\phi})$ where $\overline{u}: \Lcal \boxtimes \mathcal{O}_S^{\oplus \langle \Lcal,\alpha\rangle} \twoheadrightarrow \Fcal$ and $\overline{\phi}: \mathcal{O}_S^{\oplus \langle \Lcal', \alpha\rangle} \simeq \Hom_{\mathcal{O}_X}(\Lcal' \boxtimes \mathcal{O}_S, \Fcal)$ are canonically associated to $u$ and $\phi$ respectively.
\endproof
To sum up,  for any $\alpha$ and any pair of line bundles $\Lcal \prec \Lcal'$, $\RaLLp$ is a smooth $\GaL \times \GaLp$-scheme such that
\begin{align}\label{eq:RLLp}
\RaLLp /  \GaL  \simeq \QaLp\ , \qquad  \RaLLp /  \GaLp  \simeq \QaLLp\ .
\end{align}
Note that $\RaLLp$ is nothing but the fiber product of stacks 
\begin{align}
\RaLLp=\QaLp \underset{\bCohaLp}{\times} \QaLLp\ ,
\end{align}
and can be thought of as the fine moduli space parameterizing triples $(\Fcal,u,v)$ where $\Fcal$ is a coherent sheaf on $X$ of class $\alpha$ which is strongly generated by $\Lcal'$ and $u,v$ a pair of trivializations of $\Hom(\Lcal, \Fcal), \Hom(\Lcal',\Fcal)$ respectively. The open embedding $j_{\Lcal,\Lcal',\alpha}$ is now given by the composition
\begin{align}\label{eq:transmapscoh}
\bCohaLp=\big[\QaLp/\GaLp\big]\simeq [\RaLLp/\GaL\times \GaLp]\simeq [\RaLpL/\GaL\times \GaLp]\simeq \big[\QaLLp/\GaL\big]\hookrightarrow \bCohaL\ .
\end{align}
Thus we have a direct system $\langle \bCohaL, j_{\Lcal, \Lcal', \alpha}\rangle$ and thanks to Serre's theorem we get
\begin{align}
\bCoha \simeq \lim_{\genfrac{}{}{0pt}{}{\to}{\Lcal}} \, \bCohaL \coloneqq\lim_{\genfrac{}{}{0pt}{}{\to}{\mathfrak{Pic}(X)}} \, \bCohaL \ .
\end{align}
This provides the desired description of $\bCoha$ as an ind-algebraic stack.

Let us finish this section by recalling the structure of the (singular) cohomology ring of the stacks $\bCoha$. Here we assume that $k=\C$ and simply write $H^\ast(\bullet)$ for $H^\ast(\bullet, \Q)$. Let us fix a basis $\Pi=\{1, \pi_1, \ldots, \pi_{2g}, \varpi\}$ of $H^\ast(X)$, with $1 \in H^0(X), \pi_1, \ldots, \pi_{2g} \in H^1(X)$ and $\varpi \in H^2(X)$. For $i \in \N$, let 
\begin{align}\label{eq:chernclasses}
c_i(\Efrak_\alpha)=\sum_{\pi \in \Pi} c_{i,\pi}(\Efrak_\alpha) \otimes \pi^* \in H^\ast(\bCoha) \otimes H^\ast(X)
\end{align}
be the K\"unneth decomposition of the $i$-th Chern class of the tautological sheaf $\Efrak_\alpha$. Here $\{\pi^*\}_\pi$ stands for the dual basis with respect to the intersection form. For any (super) commutative algebra $A$ we write $S^d(A)$ for the subalgebra of $A^{\otimes d}$ fixed under the natural action of the symmetric group $\mathfrak{S}_d$.
\begin{theorem}[Heinloth, \cite{art:heinloth2012}]\label{T:Heinloth} 
The rational cohomology ring $H^\ast(\bCoha)$ is described as follows:
\begin{enumerate}\itemsep0.2cm
\item[(a)] If $\alpha=(0,d)$ then $H^\ast(\bCoha) \simeq S^d(H^\ast(X)[z])$,
\item[(b)] If $\alpha=(r,d)$ with $r >0$ then $H^\ast(\bCoha) \simeq \Q[c_{i,\pi}(\Efrak_\alpha)]_{i,\pi}$ is freely generated  (as a supercommutative algebra)
by the classes $c_{i,\pi}(\Efrak_\alpha)$ for $i \geq 2, \pi \in \Pi$ and $i=1, \pi \in \Pi \setminus \{1\}$.
\end{enumerate}
Moreover, the stack $\bCoha$ is cohomologically pure\footnote{The Hodge theory of algebraic stacks, locally of finite type, has been introduced for example in \cite[Section~2]{art:dhillon2006}.} for any $\alpha$.
\end{theorem}
By Poincar\'e duality, the assignment $c \mapsto c \cap [\bCoha]$ identifies $H^i(\bCoha)$ with $H_{-\langle \alpha,\alpha\rangle-i}(\bCoha)$, where $H_\ast(\bullet)$ stands for the Borel-Moore homology with rational coefficients. Hence the above theorem also yields a description of the Borel-Moore homology groups $H_\ast(\bCoha)$.

\subsection{Higgs sheaves}

Recall that a \emph{Higgs sheaf} on $X$ is a pair $(\Fcal,\theta)$ with $\Fcal \in \Coh(X)$ and $\theta \in \Hom(\Fcal,\Fcal \otimes \omega_X)$. We say that $(\Fcal,\theta)$ is of numerical class $\alpha=(r,d)$ if $\Fcal$ is. Higgs sheaves form the object of a Calabi-Yau two-dimensional abelian category $\Higgs(X)$, in which the Euler form and Serre duality take the following form (see e.g. \cite{art:gothenking2005}). Define, for $\underline{\Fcal}\coloneqq(\Fcal, \theta_\Fcal)$, $\underline{\Gcal}\coloneqq(\Gcal, \theta_\Gcal) \in \Higgs(X)$
\begin{align}
\langle \underline{\Fcal},\underline{\Gcal}\rangle = \dim \Hom(\underline{\Fcal},\underline{\Gcal}) -\dim \Ext^1(\underline{\Fcal},\underline{\Gcal}) + \dim \Ext^2(\underline{\Fcal},\underline{\Gcal})\ .
\end{align}
Then
\begin{align}\label{eq:eulerformhiggs}
\langle \underline{\Fcal},\underline{\Gcal} \rangle=\langle \Fcal, \mathcal{G} \rangle - \langle \Fcal, \mathcal{G}\otimes \omega_X\rangle = \langle \Fcal, \mathcal{G} \rangle +\langle\mathcal{G},  \Fcal \rangle =2(1-g)\rk(\Fcal)\rk(\mathcal{G})\ .
\end{align}
Moreover, Serre's duality holds:
\begin{align}
\Ext^i(\underline{\Fcal},\underline{\Gcal}) \simeq \Ext^{2-i}(\underline{\Gcal},\underline{\Fcal})^\ast
\end{align}
for all $i=0,1,2$. The slope of a Higgs sheaf is the slope of its underlying coherent sheaf, i.e.,
\begin{align}
\mu(\underline{\Fcal})=\mu(\Fcal)=\frac{\deg(\Fcal)}{\rk(\Fcal)} \ .
\end{align}
A Higgs sheaf is \emph{semistable} if $\mu(\underline{\Gcal}) \leq \mu(\underline{\Fcal})$ for any Higgs subsheaf $\underline{\Gcal} \subset \underline{\Fcal}$, i.e. if $\mu(\Gcal) \leq \mu(\Fcal)$ for any subsheaf $\Gcal$ of $\Fcal$ such that $\theta_\Fcal(\Gcal) \subset \Gcal \otimes \omega_X$. Semistable Higgs sheaves of fixed slope $\nu \in \Q \cup \{\infty\}$ form an abelian subcategory $\Higgs^{\mathsf{ss},\nu}(X)$ of $\Higgs(X)$, which is stable under extensions.

For a Higgs sheaf $(\Fcal, \theta)$, we denote by $\theta^{\, k}$ the composition
\begin{align}
\begin{aligned}
  \begin{tikzpicture}[xscale=3.5,yscale=-1]
  \node (A0_0) at (0.3, 0) {$\theta^{\, k}\colon \Fcal$};
\node (A1_0) at (1, 0) {$ \Fcal\otimes \omega_X$};
\node (A2_0) at (1.8, 0) {$\cdots$};
\node (A3_0) at (2.8, 0) {$\Fcal\otimes \omega_X^{\otimes\, k}$};
    \path (A0_0) edge [->]node [auto] {$\scriptstyle{\theta}$} (A1_0);
   \path (A1_0) edge [->]node [auto] {$\scriptstyle{\theta\otimes\mathsf{id}_{\omega_X}}$} (A2_0);
   \path (A2_0) edge [->]node [auto] {$\scriptstyle{\theta^{\otimes\, (k-1)}\otimes\mathsf{id}_{\omega_X}}$} (A3_0);
  \end{tikzpicture}
\end{aligned}
\end{align}
$(\Fcal, \theta)$ is called \emph{nilpotent} if there exists $s >0$ such that $\theta^{\, s}$ vanishes; we call $s$ the \emph{nilpotency index of $\theta$}. Nilpotent Higgs sheaves form an abelian subcategory $\Higgs^{\mathsf{nilp}}(X)$ of $\Higgs(X)$ which is closed under extensions, subobjects and quotients.

Let $(\Fcal, \theta)$ be a nilpotent Higgs sheaf on $X$ with nilpotency index $s$. For $k\geq 1$, define $\Fcal_k\coloneqq\mathsf{Im}\big(\theta^{\, k}\otimes \mathsf{id}_{\omega_X^{\, \otimes -k}}\big)$. Then $\Fcal_k$ is a subsheaf of $\Fcal$. Finally, set $\Fcal_0\coloneqq\Fcal$. Then there are chains respectively of inclusions and of epimorphisms
\begin{align}
\{0\}=\Fcal_s\subset \Fcal_{s-1}\subset \cdots \Fcal_1\subset \Fcal_0=\Fcal \quad\text{and}\quad \Fcal=\Fcal_0 \twoheadrightarrow \Fcal_1\otimes \omega_X\twoheadrightarrow \cdots \twoheadrightarrow \Fcal_s\otimes \omega_X^{\otimes\, s}=0 \ .
\end{align}
Define for $k\geq 0$
\begin{align}
\Fcal_k'\coloneqq\ker\big(\Fcal_k \twoheadrightarrow \Fcal_{k+1}\otimes \omega_X\big) \quad\text{and}\quad \Fcal_k''\coloneqq\Fcal_k/\Fcal_{k+1} \ .
\end{align}
Therefore we have chains respectively of inclusions and of epimorphisms
\begin{align}
\{0\}=\Fcal_s'\subset\Fcal_{s-1}' \subset \cdots \Fcal_1'\subset \Fcal_0'\quad\text{and}\quad \Fcal_0'' \twoheadrightarrow \Fcal_1''\otimes \omega_X\twoheadrightarrow \cdots \twoheadrightarrow \Fcal_m''\otimes \omega_X^{\otimes\, m}=0 \ .
\end{align}
For $k\geq 1$, set
\begin{align}
\alpha_k\coloneqq\overline{\ker\big(\Fcal_{k-1}''\otimes \omega_X^{\otimes\, k-1}\to \Fcal_{k}''\otimes \omega_X^{\otimes\, k}\big)}\ .
\end{align}
Put $\ell=\deg(\omega_X)$. One has
\begin{align}\label{eq:kernelk}
\overline{\ker\big(\theta^{\, k}\big)}=\sum_{h=1}^k\,\sum_{j=h}^s\, \alpha_j((h-j)\,\ell)
\end{align}
for any $k\geq 1$.

This computation justifies the following definition. Given $\alpha \in (\Z^2)^+$, a finite sequence $\ual=(\alpha_1, \ldots, \alpha_s)$ of elements of $(\Z^2)^+$ is called a \emph{Jordan type of class $\alpha$} if it satisfies
\begin{align}
\alpha=\sum_{i=1}^s \sum_{k=0}^{i-1} \alpha_i(-k\ell)\ ,
\end{align}
where for any $\beta=(r, d)\in (\Z^2)^+$ and $n\in \Z$, we set $\beta(n)\coloneqq(r, d+nr)$. We call $s$ the \emph{length $\ell(\ual)$} of $\ual$. Note that unless $\rk(\alpha)=0$ there are countably many Jordan types of class $\alpha$. Let $J_\alpha$ be the set of all Jordan types of class $\alpha$. 

We may helpfully represent a Jordan type by its associated colored Young diagram as follows (here with $s=4$) 
\begin{align}\label{diag:Young}
\begin{aligned}
\ytableausetup {mathmode, boxsize=3.1em,centertableaux}
\begin{ytableau}
\scriptstyle{\alpha_4}   \\
\scriptstyle\alpha_4(-\ell) & \scriptstyle\alpha_3 \\
\scriptstyle{\alpha_4(-2\ell)} &\scriptstyle \alpha_3(-\ell) & \scriptstyle\alpha_2\\
\scriptstyle{\alpha_4(-3\ell)} &\scriptstyle \alpha_3(-2\ell) & \scriptstyle\alpha_2(-\ell)&\scriptstyle\alpha_1
\end{ytableau}
\end{aligned}
\end{align}

Given a nilpotent Higgs sheaf $(\Fcal, \theta)$ of class $\alpha$, we will say that $(\Fcal, \theta)$ is \emph{of Jordan type $\ual$} if for all $k \geq 1$, the sheaf $\ker\big(\theta^{\, k}\big)$ satisfies \eqref{eq:kernelk}. In the pictorial description of $\ual$, this corresponds to the bottom $k$ rows of the Young tableaux; thus one can think of the Higgs field $\theta$ as the composition of 'going down one box' and 'tensoring by $\omega_X$').

\subsection{Stacks of Higgs sheaves}

Let us denote by $\bHiggsa$ the stack parameterizing Higgs sheaves over $X$ of class $\alpha$. Similarly, let $\bLama$ and $\bHiggs^{\mathsf{ss}}_{\alpha}$ stand for the respectively closed and open substacks of $\bHiggsa$ parametrizing respectively nilpotent and semistable Higgs sheaves\footnote{For a derived point of view to the stack of (semistable) Higgs bundles, see e.g. \cite{art:halpernleistner2016,art:ginzburgrozenblyum2017, art:portasala2019} and references therein. For the GIT approach to the construction of moduli spaces of semistable Higgs bundles, see  \cite{art:nitsure1991, art:simpson1994, art:simpson1994-II}.}. The following is well-known, see e.g. \cite{art:ginzburg2001} and \cite[Section~7]{art:casalainawise2017}.
\begin{theorem}\label{T:StackHiggs} 
The following hold:
\begin{enumerate}\itemsep0.2cm
\item[(a)] The stack $\bHiggsa$ is locally of finite type, of dimension $-2\langle \alpha,\alpha\rangle$ and is canonically isomorphic to the underived cotangent stack $T^\ast\bCoha$ of $\bCoha$:
\begin{align}
\bHiggsa \simeq T^\ast \bCoha\coloneqq\Spec\,\Sym\Hcal^0(\TCoha)\ .
\end{align}
\item[(b)] The stack $\bLama$ is a Lagrangian substack of $\bHiggsa$.
\item[(c)] The stack $\bHiggsa^{\mathsf{ss}}$ is a global quotient stack, and it is smooth if $\alpha=(r,d)$ with $\mathsf{gcd}(r,d)=1$.
\end{enumerate}
\end{theorem}
Let us denote by $r_\alpha\colon \bHiggsa \to \bCoha$ the projection map, forgetting the Higgs field.

There exists an action of the multiplicative group $T=\G_m$ on $\bHiggsa$, which at level of families reads as $z \cdot (\Fcal,\theta)\coloneqq(\Fcal,z\theta)$ for $z\in T$ and $(\Fcal, \theta)$ a flat family of Higgs sheaves. Such an action is simply the action scaling the fibers of $r_\alpha\colon T^\ast \bCoha\to \bCoha$. Indeed, $r_\alpha$ is $T$-equivariant with respect to the trivial action of $T$ on $\bCoha$.

It is known that, for $g>1$, the preimage under the projection map $r_\alpha$ of the open substack of vector bundles $\bBuna\subset \bCoha$ is irreducible (cf.\ \cite[Section~2.10]{art:beilinsondrinfeld1991}), and that the irreducible components of $\bHiggsa$ are in fact given by the Zariski closures of the substacks $r_\alpha^{-1}(\bCoha^{\mathsf{tor=d}})$ for $d \geq 0$, where $\bCoha^{\mathsf{tor=d}}$ stands for the substack of $\bCoha$ parametrizing sheaves whose torsion part is of degree $d$. We thank Jochen Heinloth for explanations concerning these facts. We will not consider these irreducible components, and instead focus on the irreducible components of $\bLama$, which we now describe explicitly following \cite{art:bozec2017}.

There is a partition $\bLama=\bigsqcup_{\ual \in J_\alpha} \bLam_{\ual}$ where $\bLam_{\ual}$ is the locally closed substack of $\bLama$ parametrizing nilpotent Higgs sheaves of Jordan type $\ual$.

As shown in the proof of \cite[Proposition~5.2]{art:mozgovoyschiffmann2017}, we have the following.
\begin{proposition}\label{prop:MS5.2} 
For any $\alpha$ and any $\ual=(\alpha_1, \ldots, \alpha_s) \in J_{\alpha}$ the morphism 
\begin{align}
\pi_{\ual}\colon \bLam_{\ual} \to& \prod_{k=1}^s \bCoh_{\alpha_k}\\
(\Fcal, \theta) \mapsto & \Big(\ker\big(\Fcal_{k-1}''\otimes \omega_X^{\otimes\, k-1}\to \Fcal_{k}''\otimes \omega_X^{\otimes\, k}\big)\Big)_k
\end{align}
is an iterated vector bundle stack morphism\footnote{See \cite[Section~3.1]{art:garcia-pradaheinlothschmitt2014} for the definition of vector bundle stack morphisms.}.
\end{proposition}
\begin{corollary}[{cf.\ \cite[Proposition~2.3 and Corollary~2.4]{art:bozec2017}}]\label{T:Bozec} 
For any $\alpha \in (\Z^2)^+$, the irreducible components of $\bLama$ are the Zariski closures $\overline{\bLam_{\ual}}$ for $\ual \in J_{\alpha}$.
\end{corollary}

Define the following partial order $\preceq$ on $J_{\alpha}$. For $\ual=(\alpha_1, \ldots, \alpha_s), \underline{\beta}=(\beta_1, \ldots, \beta_t)\in J_{\ual}$, we have $ \underline{\beta}\preceq\ual$ if and only if for any $k\geq 1$ the following inequality holds:
\begin{align}\label{eq:preceq}
\sum_{h=1}^k\,\sum_{j=h}^s\, \alpha_j((h-j)\,\ell)\leq \sum_{h=1}^k\,\sum_{j=h}^t\, \beta_j((h-j)\,\ell)\ ,
\end{align}
where the partial order $\leq$ on $(\Z^2)^+$ is the ``standard" order: for $\ual, \underline{\beta}\in (\Z^2)^+$, we have $\underline{\beta}\leq \ual$ if and only if $\ual-\underline{\beta}\in (\Z^2)^+$. Note that in particular $\underline{\beta}\preceq \ual$ implies $\ell(\underline{\beta})\leq \ell(\ual)$. The minimal element in $J_{\alpha}$ for this order is $(\alpha)$. 

By Formula \eqref{eq:kernelk}, one can reinterpret the inequality \eqref{eq:preceq} as an inequality for the classes of kernels of subsequent powers of Higgs fields associated with $\ual$ and $\underline{\beta}$. This observation together with the semicontinuity of the rank and the dimensions of cohomology groups of a coherent sheaf (cf.\ \cite[Example~12.7.2 and Theorem~12.8]{book:hartshorne1977}), imply the following result.
\begin{proposition}\label{prop:preceq}
Let $\alpha \in (\Z^2)^+$. For any $\ual\in J_{\alpha}$,
\begin{align}
\bLam_{\preceq \ual} \coloneqq \bigsqcup_{\underline{\beta}\preceq \ual} \bLam_{\underline{\beta}}
\end{align}
is a closed algebraic substack.
\end{proposition}
An important example of an irreducible component of $\bLama$ is the \emph{zero section} $\bLam_{(\alpha)}= \bCoha \subset \bLama$, obtained for the unique Jordan type of length $s=1$. It is a closed substack by the previous proposition.

\begin{remark}
By Proposition~\ref{prop:MS5.2} and Theorem~\ref{T:Heinloth}, each strata $\bLam_{\ual}$ is cohomologically pure (indeed, this property is preserved under vector bundle stack morphisms). But then $\bLam$ is itself pure since it has a locally finite partition into pure strata. Define
\begin{align}
\bLam_{\prec \ual} \coloneqq \bigsqcup_{\underline{\beta}\prec \ual} \bLam_{\underline{\beta}} \ .
\end{align}
Then $\bLam_{\preceq \ual}=\bLam_{\prec \ual}\sqcup \bLam_{\ual}$. Thanks to the purity of $\bLam_{\ual}$, one can show that also $\bLam_{\preceq \ual}$ and $\bLam_{\prec \ual}$ are pure. In addition, there are short exact sequences
\begin{align}\label{eq:filtration}
0\to H_k(\bLam_{\prec \ual})\to H_k(\bLam_{\preceq \ual}) \to H_k(\bLam_{\ual})\to 0\ .
\end{align}
Thus there is an induced filtration of $H_\ast(\bLama)$, whose associated graded is 
\begin{align}\label{eq:gradedcoh}
\mathsf{gr}(H_\ast(\bLama))=\bigoplus_{\ual \in J_{\alpha}} H_\ast(\bLam_{\ual}) \simeq \bigoplus_{\ual \in J_{\alpha}} H_\ast(\bCoh_{\alpha_1}) \otimes \cdots \otimes H_\ast(\bCoh_{\alpha_{\ell(\ual)}})\ . 
\end{align}
Such results hold in the $T$-equivariant setting as well.
\end{remark}

\subsection{Local charts of the stack of Higgs sheaves}\label{sec:statificationHiggs}

Let us now proceed with the description of the \emph{local charts} of the stacks $\bHiggsa$. Let $\Lcal$ be a line bundle on $X$. By \cite[Proposition~14.2.4]{book:laumonmoretbailly2000}, we have a cartesian diagram
\begin{align}
\begin{aligned}
  \begin{tikzpicture}[xscale=2.3,yscale=-1]
  \node (A0_0) at (0, 0) {$\mathsf{Spec}\,\mathsf{Sym}\Hcal^0(\TCohaL)$};
\node (A2_0) at (2, 0) {$\bHiggsa$};
\node (A0_2) at (0, 2) {$\bCohaL$};
\node (A2_2) at (2, 2) {$\bCoha$};
\node(A1_1) at (1,1) {$\square$};
    \path (A0_0) edge [->]node [auto] {$\scriptstyle{v_\alpha^\Lcal}$} (A2_0);
   \path (A0_0) edge [->]node [auto] {$\scriptstyle{r_\alpha^\Lcal}$} (A0_2);
   \path (A2_0) edge [->]node [auto] {$\scriptstyle{r_\alpha}$} (A2_2);
    \path (A0_2) edge [->]node [auto] {$\scriptstyle{u_\alpha^\Lcal}$} (A2_2); 
  \end{tikzpicture}
\end{aligned}
\end{align}
where the map $u_\alpha^\Lcal$ is an open embedding, and hence also $v_\alpha^\Lcal$. Define
\begin{align}
\bHiggsaL \coloneqq\mathsf{Spec}\,\mathsf{Sym}\Hcal^0(\TCohaL)\ .
\end{align}
Thus $\bHiggsaL$ is the algebraic stack parameterizing Higgs sheaves on $X$ of class $\alpha$ such that the underlying coherent sheaf is strongly generated by $\Lcal$. Such a stack can be realized as a global quotient stack. Indeed, by the explicit description \eqref{eq:cotancomplex} of $\TCohaL$, we get
\begin{align}
\bHiggsaL\simeq [T^\ast_{\GaL}\QaL/\GaL]\ , 
\end{align}
where $T^\ast_{\GaL}\QaL\coloneqq\big(\muaL\big)^{-1}(0)$ and $\muaL$ is the composition
\begin{align}
\mathsf{Spec}\, \mathsf{Sym} \Tcal_{\QaL}=T^\ast \QaL\to \mathsf{Spec}\, \mathsf{Sym}\, (\gaL\otimes \Ocal_{\QaL}) = \big(\gaL\big)^\ast\times \QaL \to \big(\gaL\big)^\ast \ .
\end{align}

Next, let us realize $\bHiggsa$ as an ind-algebraic stack. By construction (cf.\ proof of Lemma \ref{L:rall}), there are dual pairs of short exact sequences
\begin{align}
0 \to & \gaLp \to T_{([\phi],u)}\RaLLp \to \Hom(\ker(\phi),\Fcal) \to 0\ ,\\[2pt]
0 \to & \gaL \to T_{([\psi],v)}\RaLpL \to \Hom(\ker(\psi),\Fcal) \to 0 \ ,
\end{align}
and 
\begin{align}
0 \to &  \Hom(\ker(\phi),\Fcal)^\ast \to T^\ast_{([\phi],u)}\RaLLp \stackrel{\mu'}{\longrightarrow}  (\gaLp)^\ast \to 0\ , \\[2pt]
0 \to &  \Hom(\ker(\psi),\Fcal)^\ast \to T^\ast_{([\psi],v)}\RaLpL \stackrel{\mu}{\longrightarrow} (\gaL)^\ast \to 0
\end{align}
for any $([\phi],u) \in \RaLLp, ([\psi],v) \in \RaLpL$. Recall that by Lemma~\ref{L:rall} we have a canonical isomorphism $\RaLLp \simeq \RaLpL$ as $\GaL \times \GaLp$-schemes, which sends $([\phi], u)$ to $([\overline{u}], \overline{\phi})$. The moment map relative to the Hamiltonian $\GaL \times \GaLp$-action on $T^\ast\RaLLp$ is given by
\begin{align}
\mu_\alpha^{\Lcal,\Lcal'}\coloneqq\mu'\oplus \mu \colon T^\ast_{(\phi,u)}\RaLLp \to (\gaLp)^\ast  \oplus (\gaL)^\ast\ .
\end{align}
The above complex is quasi-isomorphic to both
\begin{align}
\big[T^\ast_{[\psi]}\QaLLp \stackrel{\mu_\alpha^\Lcal}{\to} (\gaL)^\ast\big] \simeq \big[\Hom(\ker(\phi),\Fcal) \to (\gaL)^\ast\big]\ ,\\[2pt]
\big[T^\ast_{[\phi]}\QaLp \stackrel{\mu_\alpha^{\Lcal'}}{\to} (\gaLp)^\ast\big] \simeq \big[\Hom(\ker(\psi),\Fcal) \to (\gaLp)^\ast\big]\ .
\end{align}
It follows that the projection maps $T^\ast_{\GaL \times \GaLp}\RaLLp \to T^\ast_{\GaL}\QaLLp$ and $T^\ast_{\GaL \times \GaLp}\RaLLp \to T^\ast_{\GaLp}\QaLp$ are respectively principal $\GaLp$ and $\GaL$-bundles. Hence the open embedding $j_{\Lcal, \Lcal', \alpha}\colon \bCohaLp \hookrightarrow \bCohaL$ lifts to an open embedding $h_{\Lcal, \Lcal', \alpha}\colon \bHiggsaLp\hookrightarrow \bHiggsaL$ obtained as a composition
\begin{align}\label{eq:haLLp}
\begin{aligned}
\bHiggsaLp=&[T^\ast_{\GaLp}\QaLp/\GaLp]\simeq [T^\ast_{\GaL\times \GaLp} \RaLLp/\GaL\times \GaLp]\\[2pt]
&\simeq  [T^\ast_{\GaL\times \GaLp} \RaLpL/\GaL\times \GaLp]\simeq [T^\ast_{\GaL}\QaLLp/\GaL]\hookrightarrow \bHiggsaL \ ,
\end{aligned}
\end{align}
where the last morphism is obtained by applying base change with respect to $j_{\Lcal, \Lcal', \alpha}$ and \cite[Proposition~14.2.4]{book:laumonmoretbailly2000}. Thus, we obtain a directed system $\langle \bHiggsaL, h_{\Lcal, \Lcal', \alpha}\rangle$ and thus we get
\begin{align}
\bHiggsa \simeq \lim_{\genfrac{}{}{0pt}{}{\to}{\Lcal}} \, \bHiggsaL \coloneqq\lim_{\genfrac{}{}{0pt}{}{\to}{\mathfrak{Pic}(X)}} \, \bHiggsaL \ .
\end{align}

\bigskip\section{Definition of the cohomological Hall algebras}\label{sec:COHAHiggs}

The present section is devoted to the construction of $A$-homological Hall algebras associated with the 2-Calabi-Yau category $\Higgs(X)$, and their variants for the category of nilpotent Higgs sheaves and the category of semistable Higgs bundles of fixed slope. 

\subsection{Borel-Moore homology theories}\label{sec:BM}

Although most of our results here concern the case of the cohomological Hall algebra for Borel-Moore homology or Chow groups, our constructions make sense for an arbitrary oriented Borel-Moore homology theory (OBM). Let $\mathsf{Sch}/ k$ be the category of separated $k$-schemes of finite type. Recall that an OBM theory on $\mathsf{Sch}/ k$ is the data of 
\begin{enumerate}\itemsep0.2cm
\item[(a)] for every $k$-scheme $X$, a graded abelian group $A_\ast(X)$;
\item[(b)] for every projective morphism $f\colon X \to Y$, a homomorphism $f_\ast\colon A_\ast(X) \to A_\ast(Y)$;
\item[(c)] for every locally complete intersection (l.c.i.) morphism $g\colon X \to Y$ of relative dimension $d$, a homomorphism $g^\ast\colon A_\ast(Y) \to A_{\ast+d}(X)$;
\item[(d)] an element $\mathbf{1} \in A_0(\pt)$ and for every pair $(X,Y)$ of $k$-schemes, a bilinear pairing $\times \colon A_\ast(X) \otimes A_\ast(Y) \to A_\ast(X \times Y)$ which is associative, commutative and for which $\mathbf{1}$ is a unit,
\end{enumerate}
satisfying a certain number of natural axioms, see \cite{book:levinemorel2007}. Of particular importance to us will be the existence of \emph{refined Gysin pullback morphisms} in the following context: let $f\colon Y \to X$ be an l.c.i morphism and let $g \colon Z \to X$ be an arbitrary morphism; then there exists a pullback morphism 
\begin{align}
f^!=f^!_g\colon A_\ast(Z) \to A_\ast(Z \times_X Y)
\end{align}
which coincides with the usual pullback morphism $f^\ast$ if $Z=X$ and $g=\mathsf{id}_X$.

In the following, we shall restrict ourselves to those OBMs which are \emph{free}, see e.g. \cite[Chapter~4]{book:levinemorel2007}. Examples of free OBM theories include K-theory and Chow groups. Although usual Borel-Moore homology is not \emph{per se} an OBM theory because of the presence of odd degree classes, it satisfies all the important properties and all of our constructions will be valid in that situation as well. We refer to the papers \cite[Appendix~A]{art:minets2018} and \cite[Section~1.1]{art:yangzhao2014} which contain all the properties of OBM theories which we will need here.

Let $G$ be a reductive algebraic group, a $G$-equivariant version of OBMs has been defined in \cite{art:hellermalagonlopez2013}, and therefore there exists a theory of OBMs for global quotient stacks. In  \cite{art:kresch1999}, Kresch define a theory of Chow groups for algebraic stacks, locally of finite type (and stratified by global quotient stacks).

Since, we are considering $\bHiggsa$ as an ind-algebraic stack, we give here the following definition. Let $\Xscr$ be an algebraic stack such that $\displaystyle \Xscr\simeq \lim_{\to}\, \Uscr_i$, where the limit is taken with respect to the directed system $(\Uscr_i, \jmath_{i\leq i'}\colon \Uscr_i\to \Uscr_{i'})$, formed by all open substacks $\Uscr_i$ of $\Xscr$ of finite type and open immersions $\jmath_{i\leq i'}$. We define
\begin{align}
A_\ast(\Xscr):=\lim_{\genfrac{}{}{0pt}{}{\longleftarrow}{\Uscr_i}}\, A_\ast(\Uscr_i)\ .
\end{align}
This graded abelian group tends to be very large and untractable. We also define a smaller graded abelian group $A^0_\ast(\Xscr) \subset A_\ast(\Xscr)$ as follows. First, we say that an algebraic stack $\Zscr$, such that $\displaystyle \Zscr\simeq \lim_{\to}\, \Wscr_i$, is \emph{admissible} if $\mathsf{codim}(\Zscr\setminus \Wscr_i)\to \infty$ if $i\to \infty$. Let $A^0_\ast(\Xscr)$ be the subgroup of $A_\ast(\Xscr)$ consisting of classes supported on an admissible closed substack, i.e., lying in the image of the direct image morphism $i_\ast\colon A_\ast(\Zscr) \to A_\ast(\Xscr)$ for $i\colon \Zscr\hookrightarrow \Xscr$ an admissible closed substack. The subgroups $A^0_\ast$ are preserved under the standard operations $(-)_\ast,(-)^\ast, (-)^!$ with respect to morphisms which are of finite relative dimension.

\subsection{Cohomological Hall algebra of the stacks of coherent sheaves}\label{sec:COHAcoh}

Let us first define the cohomological Hall algebra attached to the simpler $1$-dimensional category $\Coh(X)$. For this, we need to introduce stacks classifying extensions between coherent sheaves. For $\alpha,\beta \in (\Z^2)^+$, let $\bCohtildeab$ be the stack parameterizing inclusions $\Gcal\subset \Fcal$, where $\Gcal$ is a flat family of coherent sheaves of class $\beta$ and $\Fcal$ is a flat family of coherent sheaves of class $\alpha+\beta$. It is a smooth irreducible algebraic stack, locally of finite type, of dimension $-\langle \alpha,\alpha\rangle -\langle \beta,\beta\rangle -\langle \alpha,\beta\rangle$, equipped with a pair of morphisms
\begin{align}\label{eq:convoldiagramcoh}
\begin{aligned}
  \begin{tikzpicture}[xscale=3,yscale=-1]
  \node (A0_0) at (0, 0) {$\bCoha \times\bCohb$};
\node (A1_0) at (1, 0) {$\bCohtildeab$};
\node (A2_0) at (2, 0) {$\bCohab$};
    \path (A1_0) edge [->]node [above] {$\scriptstyle{q_{\alpha, \beta}}$} (A0_0);
   \path (A1_0) edge [->]node [above] {$\scriptstyle{p_{\alpha,\beta}}$} (A2_0);
  \end{tikzpicture}
\end{aligned}
\end{align}
defined at the level of flat families by $q_{\alpha,\beta}(\Gcal\subset\Fcal)=(\Fcal/\Gcal,\Gcal)$ and $p_{\alpha,\beta}(\Gcal\subset\Fcal)=\Fcal$. The map
$q_{\alpha,\beta}$ is a vector bundle stack morphism, while $p_{\alpha,\beta}$ is a proper representable morphism (see e.g. \cite[Section~3.1]{art:garcia-pradaheinlothschmitt2014}). The stack $\bCohtildeab$ is equipped with a short exact sequence of \emph{tautological sheaves} 
\begin{align}
0\to \widetilde \Efrak_\beta\to \widetilde \Efrak_{\alpha+\beta}\to \widetilde \Efrak_\alpha \to 0 \ ,
\end{align}
where $\widetilde \Efrak_\alpha, \widetilde \Efrak_{\alpha+\beta}, \widetilde \Efrak_\beta\in \Coh\big(\bCohtildeab\times X
\big)$. Moreover, the following relations hold between tautological sheaves
\begin{align}
\big(q_{\alpha, \beta}\big)^\ast \mathsf{pr}_\alpha^\ast\Efrak_\alpha \simeq \widetilde \Efrak_\alpha \ , \quad \big(q_{\alpha, \beta}\big)^\ast \mathsf{pr}_\alpha^\ast \Efrak_\beta \simeq \widetilde \Efrak_\beta \quad\text{and}\quad \big(p_{\alpha,\beta}\big)^\ast \Efrak_{\alpha+\beta} = \widetilde \Efrak_{\alpha+\beta}\ ,
\end{align}
where $\mathsf{pr}_\alpha^\ast, \mathsf{pr}_\beta^\ast$ are the projections from $\bCoha\times \bCohb$ to the two factors respectively.
\begin{definition} 
Let $A$ be either Borel-Moore homology or a free oriented Borel-Moore homology theory. The \emph{$A$-homological Hall algebra} of the category $\Coh(X)$
is the abelian group
\begin{align}
\COHA_{\Coh(X)}\coloneqq\bigoplus_{\alpha \in (\Z^2)^+} A_\ast(\bCoha)
\end{align}
equipped with the multiplication
\begin{align}
A_\ast(\bCoha) \otimes A_\ast(\bCohb) \to A_*(\bCohab)\ , \quad c_1 \otimes c_2 \mapsto (p_{\alpha,\beta})_\ast (q_{\alpha,\beta})^\ast(c_1 \boxtimes c_2)\ .
\end{align}
\end{definition}
Since $\bCoha, \bCohb$ and $\bCohtildeab$ are smooth, it is easy to see that $\COHA_{\Coh(X)}$ is a graded associative algebra (see \cite{art:schiffmannvasserot2018} where the case of $A=H_\ast$ is considered).  Although this is already a very interesting and still mysterious algebra, the aim of this paper is to study its \emph{two}-dimensional counterpart, defined using the moduli stacks of Higgs sheaves on $X$.

\subsubsection{Local presentation of the diagram \eqref{eq:convoldiagramcoh}}

To unburden the notation, let us set $\VV_{\Lcal,\gamma}=k^{\, \langle \Lcal, \gamma \rangle}$ for any $\Lcal$ and any $\gamma \in (\Z^2)^+$. Let us also fix an isomorphism $\VV_{\Lcal,\,\alpha+\beta} \simeq \VV_{\Lcal,\,\alpha} \oplus \VV_{\Lcal,\, \beta}$. 

\begin{lemma}\label{lem:standard} 
Let $\Lcal'$ be a line bundle on $X$ and let $\alpha,\beta \in (\Z^2)^+$. Then there exists $N \ll 0$ such that, for any line bundle $\Lcal$ with $\deg(\Lcal) \leq N$, for any scheme $S$, and any $(\Gcal\subset \Fcal) \in \bCohtildeab(S)$ with $\Fcal \in \bCohabLp(S)$ we have $(\Fcal/\Gcal, \Gcal) \in \bCohaL(S)\times \bCohbL(S)$.
\end{lemma}
\begin{proof} 
This comes from the fact that $\bCoh^{>\Lcal'}_{\alpha+\beta}$ is a global quotient stack, in particular of finite type, and $\displaystyle\bCoh_{\gamma}\simeq \lim_{\genfrac{}{}{0pt}{}{\to}{\Lcal}}\, \bCoh^{>\Lcal}_{\gamma}$ for any $\gamma\in  (\Z^2)^+$. See e.g. \cite[Lemma~2.3]{art:schiffmann2004} for details.
\end{proof}
As a corollary, we see that for any $\alpha,\beta,\Lcal'$ as above, and any $\Lcal$ of sufficiently negative degree we have
\begin{align}
p_{\alpha,\beta}^{-1}(\bCohabLp) \subset q_{\alpha,\beta}^{-1}(\bCohaL \times \bCohbL)\ .
\end{align}

Now we provide a presentation of $\bCohtildeab$ as an ind-algebraic stack. Define the subscheme of $\QtildeabL$ consisting of points $\big[\phi \colon \Lcal\otimes\VV_{\Lcal,\, \alpha+ \beta}\twoheadrightarrow \Fcal\big]$ for which $\phi(\Lcal \otimes \VV_{\Lcal,\,\beta})$ is strongly generated by $\Lcal$, and
\begin{align}
\QtildeabLLp\coloneqq\QtildeabL \cap \QabLLp \subset \QabL\ .
\end{align}
Note that $\QtildeabL$ may not be closed in $\QabL$ (because of being strongly generated is an open condition), but by Lemma \ref{lem:standard} $\QtildeabLLp$ is closed in $\QabLLp$. Let $\PabL\subset\GabL$ be the group consisting of $g\in \GabL$ such that $g(\VV_{\Lcal, \, \beta})=\VV_{\Lcal, \, \beta}$. Define the global quotient stack
\begin{align}
\bCohtildeabLLp\simeq \big[\QtildeabLLp/\PabL\big]\ .
\end{align}
It is the algebraic stack parameterizing extensions $0\to \Gcal\to \Fcal\to \Ecal \to 0$ of coherent sheaves on $X$ with $\overline{\Gcal}=\beta$, $\overline{\Ecal}=\alpha$ and $\Fcal$ strongly generated by $\Lcal'$, while the sheaves $\Hcal$ and $\Gcal$ are strongly generated by $\Lcal$. Moreover, we have an open embedding $u_{\alpha,\beta}^{\Lcal, \Lcal'}\colon \bCohtildeabLLp\to \bCohtildeab$.

Let $\alpha, \beta, \Lcal'$ as above and consider two line bundles $\Lcal_1, \Lcal_2$ of degrees less than or equal to $N$. By Lemma \ref{lem:standard}, we have an isomorphism $\bCohtildeabLoneLp\simeq \bCohtildeabLtwoLp$. An explicit realization of such an isomorphism can be given by following the same reasoning as in Section \ref{sec:stackcoherentsheaves}. 

Let $\RtildeabLoneLtwoLp$ be the scheme representing the contravariant functor $\mathsf{Aff}/k \to (\mathsf{Sets})$ which assigns to an affine $k$-variety $S$ the set of pairs $([\phi], u)$, where $[\phi]\in \QtildeabLoneLp(S)$, and $u\colon \VV_{\Lcal_2,\,\alpha+\beta}\boxtimes \Ocal_S\simeq \Hom(\Lcal_2\boxtimes\Ocal_S,\Fcal)$ is an isomorphism such that $u(\VV_{\Lcal_2, \,\beta}\boxtimes\Ocal_S)=\Hom(\Lcal_2\boxtimes\Ocal_S,\Gcal)$. One can show that $\RtildeabLoneLtwoLp$ is a $\PabLtwo$-principal bundle over $\QtildeabLoneLp$ and a $\PabLone$-principal bundle over $\QtildeabLtwoLp$. Therefore, at the level of global quotient stacks we have
\begin{align}
\bCohtildeabLoneLp=[\QtildeabLoneLp/\PabLone]\simeq [\RtildeabLoneLtwoLp/\PabLone\times \PabLtwo]\simeq [\QtildeabLtwoLp/\PabLtwo]=\bCohtildeabLtwoLp\ .
\end{align}
On the other hand, given two line bundles $\Lcal_1$ and $\Lcal_2$ such that $\Lcal_1\preceq \Lcal_2$ and given $\Lcal$ of sufficiently negative degree (less than the $N$'s of Lemma \ref{lem:standard} both for the triple $\alpha, \beta, \Lcal_1$ and the triple $\alpha, \beta, \Lcal_2$), we have an open embedding $\jmath_{\alpha,\beta}^{\Lcal, (\Lcal_1, \Lcal_2)}\colon \bCohtildeabLLone\to \bCohtildeabLLtwo$. We get a direct system $\langle \bCohtildeabLLp, \jmath_{\alpha,\beta}^{\Lcal, (\Lcal_1, \Lcal_2)}\rangle$ and therefore
\begin{align}
\bCohtildeab \simeq \lim_{\genfrac{}{}{0pt}{}{\to}{\Lcal'}} \, \bCohtildeabLLp :=\lim_{\genfrac{}{}{0pt}{}{\to}{\widetilde{\mathfrak{Pic}}(X)}} \, \bCohtildeabLLp \ .
\end{align} 
There are maps 
\begin{align}\label{eq:overlinemaps}
\overline{p}_{\alpha,\beta}^{\, \Lcal, \Lcal'}\colon \QtildeabLLp \hookrightarrow \QabLLp\qquad\text{and}\qquad\overline{q}_{\alpha,\beta}^{\, \Lcal, \Lcal'}\colon\QtildeabL \to \QaL \times \QbL\ ,
\end{align}
which are respectively a proper morphism and an affine fibration of rank $\langle \Lcal,\beta\rangle \langle \Lcal,\alpha\rangle-\langle \beta,\alpha\rangle$. Moreover, $H_{\alpha, \beta}^{\Lcal}\coloneqq\GaL \times \GbL$ acts on $\QaL \times \QbL$, $\PabL$ acts on $\QtildeabLLp$ while
$\GabL$ acts on $\QabLLp$. Note that $H_{\alpha, \beta}^{\Lcal} \subset \PabL \subset \GabL$ is the inclusion of a Levi factor of a parabolic subgroup of $\GabL$, and that $\overline{p}_{\alpha,\beta}^{\, \Lcal, \Lcal'}, \overline{q}_{\alpha,\beta}^{\, \Lcal, \Lcal'}$ are $\PabL$-equivariant (with respect to the canonical maps $\PabL \to H_{\alpha, \beta}^{\Lcal}$ and $\PabL \to \GabL$). Thus, we have induced morphisms at the stacky level
\begin{align}
p_{\alpha,\beta}^{\Lcal, \Lcal'}\colon \bCohtildeabLLp \to \bCohabLp \quad\text{and}\quad q_{\alpha,\beta}^{\Lcal, \Lcal'}\colon \bCohtildeabLLp \to \bCohaL \times \bCohbL\ ,
\end{align}
which fit into, respectively, the cartesian diagram
\begin{align}
\begin{aligned}
  \begin{tikzpicture}[xscale=2.3,yscale=-1]
  \node (A0_0) at (0, 0) {$\bCohtildeabLLp$};
\node (A2_0) at (2, 0) {$\bCohabLp$};
\node (A0_2) at (0, 2) {$\bCohtildeab$};
\node (A2_2) at (2, 2) {$\bCohab$};
\node(A1_1) at (1,1) {$$};
    \path (A0_0) edge [->]node [auto] {$\scriptstyle{p_{\alpha,\beta}^{\Lcal, \Lcal'}}$} (A2_0);
   \path (A0_0) edge [->]node [auto] {$ $} (A0_2);
   \path (A2_0) edge [->]node [auto] {$ $} (A2_2);
    \path (A0_2) edge [->]node [auto] {$\scriptstyle{p_{\alpha,\beta}}$} (A2_2); 
  \end{tikzpicture}
\end{aligned}
\end{align}
and the commutative diagram
\begin{align}
\begin{aligned}
  \begin{tikzpicture}[xscale=2.3,yscale=-1]
  \node (A0_0) at (0, 0) {$\bCohtildeabLLp$};
\node (A2_0) at (2, 0) {$\bCohaL\times \bCohbL$};
\node (A0_2) at (0, 2) {$\bCohtildeab$};
\node (A2_2) at (2, 2) {$\bCoha\times\bCohb$};
\node(A1_1) at (1,1) {$$};
    \path (A0_0) edge [->]node [auto] {$\scriptstyle{q_{\alpha,\beta}^{\Lcal, \Lcal'}}$} (A2_0);
   \path (A0_0) edge [->]node [auto] {$ $} (A0_2);
   \path (A2_0) edge [->]node [auto] {$ $} (A2_2);
    \path (A0_2) edge [->]node [auto] {$\scriptstyle{q_{\alpha,\beta}}$} (A2_2); 
  \end{tikzpicture}
\end{aligned}
\end{align}

\subsection{Cohomological Hall algebra of the stack of Higgs sheaves}


There are again maps
\begin{align}\label{eq:convoldiagramhiggs}
\begin{aligned}
  \begin{tikzpicture}[xscale=3,yscale=-1]
  \node (A0_0) at (0, 0) {$\bHiggsa \times\bHiggsb$};
\node (A1_0) at (1, 0) {$\bHiggstildeab$};
\node (A2_0) at (2, 0) {$\bHiggsab$};
    \path (A1_0) edge [->]node [above] {$\scriptstyle{q_{\alpha, \beta}}$} (A0_0);
   \path (A1_0) edge [->]node [above] {$\scriptstyle{p_{\alpha,\beta}}$} (A2_0);
  \end{tikzpicture}
\end{aligned}
\end{align}
defined by $q_{\alpha,\beta}(\underline{\Gcal} \subset \underline{\Fcal})=(\underline{\Fcal}/\underline{\Gcal},\underline{\Gcal})$ and $p_{\alpha,\beta}(\underline{\Gcal} \subset \underline{\Fcal})=\underline{\Fcal}$. We use the same notation as in Section \ref{sec:COHAcoh}, hoping no confusion may arise. 

The map $p_{\alpha,\beta}$ is still a proper representable morphism, but the map $q_{\alpha,\beta}$ is very far from being a vector bundle stack morphism, or even an l.c.i. morphism hence it is not possible to define directly a pullback morphism $q_{\alpha,\beta}^\ast\colon A_\ast(\bHiggsa \times \bHiggsb) \to A_\ast(\bHiggstildeab)$. In order to circumvent this difficulty, we follow \cite[Section~4]{art:schiffmannvasserot2013-II} (see also \cite{art:yangzhao2014} for the case of arbitrary Borel-Moore homology theories) and embed the convolution diagram \eqref{eq:convoldiagramhiggs} into a convolution diagram of smooth varieties and use refined Gysin pullbacks. One caveat of this approach is that we only manage to construct this embedding locally, and hence work with local atlases. The case of rank zero Higgs stacks is studied in details in \cite{art:minets2018}: since rank zero Higgs stacks are global quotient stacks, the author applies directly the machinery of \cite{art:schiffmannvasserot2013-II,art:yangzhao2014}. We shall follow closely \cite{art:minets2018} in some of the arguments here.

\subsubsection{Local charts of the stack $\bHiggstildeab$}\label{sec:localdescriptionHiggstilde}

Let $\alpha, \beta \in (\Z^2)^+$ and $\Lcal'$ be a line bundle on $X$. Let $\Lcal$ be a line bundle of degree less or equal that the $N$, depending on $\alpha, \beta, \Lcal'$ of Lemma \ref{lem:standard}. The diagram \ref{eq:convoldiagramcoh} reduces locally to
\begin{align}\label{eq:convoldiagramcohLLp}
\begin{aligned}
  \begin{tikzpicture}[xscale=3.8,yscale=-1]
  \node (A0_0) at (0, 0) {$\bCohaL \times\bCohbL$};
\node (A1_0) at (1, 0) {$\bCohtildeabLLp$};
\node (A2_0) at (1.8, 0) {$\bCohabLp$};
    \path (A1_0) edge [->]node [above] {$\scriptstyle{q_{\alpha, \beta}^{\Lcal, \Lcal'}}$} (A0_0);
   \path (A1_0) edge [->]node [above] {$\scriptstyle{p_{\alpha,\beta}^{\Lcal, \Lcal'}}$} (A2_0);
  \end{tikzpicture}
\end{aligned}
\end{align}
Such a diagram can be realized at the level of atlases in the following way. To unburden the notation, set  
\begin{align}
H\coloneqq H_{\alpha, \beta}^{\Lcal} \ ,\ P\coloneqq \PabL\ ,\ G\coloneqq \GabL\ ,\ V\coloneqq\QtildeabL \ , \ V^\circ\coloneqq \QtildeabLLp\ , \ Y\coloneqq \QaL\times \QbL\ , \ X'\coloneqq \QabLLp
\end{align}
and
\begin{align}
W\coloneqq G\underset{P}{\times} V\ , \quad W^{\circ}\coloneqq G\underset{P}{\times} V^{\circ}\ , \quad X\coloneqq G\underset{P}{\times} Y\ .
\end{align}
Since $\bCohaL\times \bCohbL\simeq [X/ G]$, $\bCohabLp\simeq [X'/G]$ and $\bCohtildeabLLp\simeq [W^\circ/G]$, the diagram \eqref{eq:convoldiagramcohLLp} corresponds to the following diagram of $G$-varieties
\begin{align}
\begin{aligned}
  \begin{tikzpicture}[xscale=1.8,yscale=-1.3]
  \node (A0_0) at (0.2, 1) {$X$};
\node (A1_0) at (1, 0) {$W$};
\node (A2_0) at (2, 0) {$W^\circ$};
\node (A3_0) at (2.8, 1) {$X'$};
    \path (A2_0) edge [->]node [above] {$\scriptstyle{i}$} (A1_0);
    \path (A1_0) edge [->]node [above] {$\scriptstyle{f}$} (A0_0);
   \path (A2_0) edge [->]node [above] {$\scriptstyle{g}$} (A3_0);
  \end{tikzpicture}
\end{aligned}
\end{align}
where $i\colon W^{\circ} \to W$ is the open immersion, and where $f,g$ are defined respectively by
\begin{align}
f\colon (h,v)\bmod{P} \mapsto (h,\overline{q}_{\alpha,\beta}^{\,\Lcal, \Lcal'}(v))\bmod{P}\quad\text{and}\quad g\colon (h,v)\bmod{P} \mapsto h \cdot \overline{p}_{\alpha,\beta}^{\,\Lcal, \Lcal'}(v)\ .
\end{align}
Since $(f\circ i,g)$ is a regular embedding, we can identify $W^\circ$ with a smooth subvariety of $X \times X'$. Put $Z^\circ\coloneqq T^\ast_{W^\circ}(X \times X')$. Denoting by $\Phi$ and $\Psi$ the projections on factors, we obtain a diagram
\begin{align}
\begin{aligned}
  \begin{tikzpicture}[xscale=2,yscale=-1.3]
  \node (A0_0) at (0.2, 1) {$T^\ast X$};
\node (A1_0) at (1, 0) {$Z^\circ $};
\node (A2_0) at (1.8, 1) {$T^\ast X'$};
    \path (A1_0) edge [->]node [above] {$\scriptstyle{\Phi}$} (A0_0);
   \path (A1_0) edge [->]node [above] {$\scriptstyle{\Psi}$} (A2_0);
  \end{tikzpicture}
\end{aligned}
\end{align}
Note that $\Psi$ is proper since $\overline{p}_{\alpha,\beta}^{\,\Lcal, \Lcal'}$ is a closed embedding (see e.g. \cite[Lemma~2.3]{art:schiffmannvasserot2012}), while $\Phi$ is a regular morphism as both $Z^\circ$ and $T^\ast X$ are smooth. Next, set $Z^\circ_G\coloneqq Z^\circ \cap (T^\ast_GX \times T^\ast_GX')$. Then by \emph{loc. cit.} $\Phi^{-1}(T^\ast_G X) = Z^\circ_G$ and $\Psi(Z^\circ_G) \subseteq T^\ast_GX'$. We arrive at the following diagram
\begin{align}
\begin{aligned}
  \begin{tikzpicture}[xscale=2,yscale=-2]
  \node (A0_1) at (0.2, 1) {$T^\ast_G X$};
\node (A1_0) at (1, 0) {$Z^\circ_G$};
\node (A2_1) at (1.8, 1) {$T^\ast_G X'$};
 \node (A0_2) at (0.2, 2) {$T^\ast X$};
\node (A1_1) at (1, 1) {$Z^\circ $};
\node (A2_2) at (1.8, 2) {$T^\ast X'$};
    \path (A1_0) edge [->]node [above] {$\scriptstyle{}$} (A1_1);
    \path (A0_1) edge [->]node [above] {$\scriptstyle{}$} (A0_2);
   \path (A2_1) edge [->]node [above] {$\scriptstyle{}$} (A2_2);  
    \path (A1_0) edge [->]node [above=0.7em, left=0.2em] {$\scriptstyle{\Phi_G}$} (A0_1);
   \path (A1_0) edge [->]node [above=0.7em, right=0.2em] {$\scriptstyle{\Psi_G}$} (A2_1);
    \path (A1_1) edge [->]node [above=0.7em, left=0.2em] {$\scriptstyle{\Phi}$} (A0_2);
   \path (A1_1) edge [->]node [above=0.7em, right=0.2em] {$\scriptstyle{\Psi}$} (A2_2);
  \end{tikzpicture}
\end{aligned}
\end{align}
in which the left square is cartesian. By Section \ref{sec:statificationHiggs}, we get $\bHiggsaL\times \bHiggsbL\simeq [T^\ast_G X/G]$ and $\bHiggsabLp\simeq [T^\ast_G X'/G]$. 
In addition, by following some of the arguments in the proof of \cite[Lemma~2.1]{art:minets2018} one can show that $[Z^\circ_G/G]$ is the stack $\bHiggstildeabLLp$ parameterizing inclusions of Higgs sheaves $\underline{\Gcal} \subset \underline{\Fcal}$ with $\underline{\Gcal} \in \bHiggsbL$ and $\underline{\Fcal} \in \bHiggsabLp$. Therefore, we obtain a diagram
\begin{align}\label{eq:convoldiagramhiggsLLp}
\begin{aligned}
  \begin{tikzpicture}[xscale=4,yscale=-1]
  \node (A0_0) at (0, 0) {$\bHiggsaL \times\bHiggsbL$};
\node (A1_0) at (1, 0) {$\bHiggstildeabLLp$};
\node (A2_0) at (1.8, 0) {$\bHiggsabLp$};
    \path (A1_0) edge [->]node [above] {$\scriptstyle{q_{\alpha, \beta}^{\Lcal, \Lcal'}}$} (A0_0);
   \path (A1_0) edge [->]node [above] {$\scriptstyle{p_{\alpha,\beta}^{\Lcal, \Lcal'}}$} (A2_0);
  \end{tikzpicture}
\end{aligned}
\end{align}
which is the local version of the diagram \eqref{eq:convoldiagramhiggs}.

\subsubsection{Definition of the multiplication}

Now we shall define for any free OBM theory $A_\ast$ (or for the usual Borel-Moore homology $H_\ast$) a map
\begin{align}\label{eq:defproduct}
m_{\alpha,\beta}^{\Lcal,\Lcal'}\colon A (\bHiggsaL)\otimes A(\bHiggsbL)\to A(\bHiggsabLp)\ .
\end{align}
In the above, we have suppressed the grading index $\ast$ in $A_\ast$ for extra readability. 

We shall continue to use the notation of the previous section. Because $\Phi$ is regular (in particular, l.c.i.) there is a refined Gysin pullback morphism
\begin{align}\label{eq:refGysin1}
\Phi^!\colon A^G(T^\ast_G X) \to A^G(Z^\circ_G)\ ,
\end{align}
and a pushforward morphism
\begin{align}\label{eq:refpush}
\Psi_{G, \ast}\colon  A^G(Z^\circ_G) \to A^G(T^\ast_GX')\ . 
\end{align}
Now, from the isomorphisms \eqref{eq:RLLp} we get the chain of isomorphisms 
\begin{align}
A^G(T^\ast_G X')=A^{\GabL}\big(T^\ast_{\GabL} \QabLLp\big)\simeq A^{\GabL\times \GabLp}\big(T^\ast_{\GabL\times \GabLp} \RabLLp\big)\simeq A^{\GabLp}\big(T^\ast_{\GabLp} \QabLp\big)\ .
\end{align}
In addition, since $T^\ast_G X =G \underset{P}{\times} T^\ast_H Y$,  by \cite[Proposition~A.6]{art:minets2018} we have  
\begin{align}
A^{\GaL \times \GbL}(T^\ast_{\GaL \times \GbL}(\QaL \times \QbL))=A^H(T^\ast_H Y)\simeq A^G(T^\ast_G X) \ .
\end{align}
By compositing all these maps together, we get
\begin{align}
A^{\GaL \times \GbL}(T^\ast_{\GaL \times \GbL}(\QaL \times \QbL))=&A^H(T^\ast_H Y)\simeq A^G(T^\ast_G X)\stackrel{\Phi^!}{\longrightarrow} A^G(Z^\circ_G)\\
&\stackrel{\Psi_{G, \ast}}{\longrightarrow} A^G(T^\ast_G X')\simeq A^{\GabLp}\big(T^\ast_{\GabLp} \QabLp\big)\ ,
\end{align}
whose restriction
\begin{align}\label{eq:defproduct2}
m_{\alpha,\beta}^{\Lcal,\Lcal'}\colon A^{\GaL}(T^\ast_{\GaL}\QaL) \otimes A^{\GbL}(T^\ast_{\GbL}( \QbL)) \to A^{\GabLp}(T^\ast_{\GabLp}\QabLp)
\end{align}
gives \eqref{eq:defproduct}.

\subsection{Main theorem} 

\begin{theorem}\label{theorem:defproduct} 
The collection of morphisms $m_{\alpha,\beta}^{\Lcal,\Lcal'}$ give rise to a canonically defined morphism
\begin{align}
m_{\alpha,\beta}\colon A(\bHiggsa) \otimes A(\bHiggsb) \to A(\bHiggsab)\ .
\end{align}
Equipped with these morphisms, 
\begin{align}
\COHA'_{\Higgs(X)}\coloneqq\bigoplus_{\alpha \in (\Z^2)^+} A(\bHiggsa)
\end{align}
is an associative algebra.
\end{theorem}
\begin{proof}
The first statement boils down to the following. Let $\alpha,\beta$ be fixed, and let $\Lcal_i$ for $i=1,2,3$ be line bundles on $X$. Assume that $\Lcal_3$ is strongly generated by $\Lcal_2$, itself strongly generated by $\Lcal_1$. Assume in addition that the conclusion of Lemma \ref{lem:standard} holds for the pairs $(\Lcal_1,\Lcal_2)$ and $(\Lcal_2,\Lcal_3)$. Then, denoting by $\mathsf{res}_{\gamma}^{\Lcal,\Lcal'} \colon A(\bHiggs_{\gamma}^{>\Lcal}) \to  A(\bHiggs_{\gamma}^{>\Lcal'})$ the pullback induced by the open embedding $h_{\Lcal, \Lcal', \gamma}$ introduced in \eqref{eq:haLLp} for any $\gamma\in (\Z^2)^+$, we have to show that
\begin{align}\label{eq:proof1}
m_{\alpha,\beta}^{\Lcal_1,\Lcal_3} &= m_{\alpha,\beta}^{\Lcal_2,\Lcal_3} \circ (\mathsf{res}^{\Lcal_1,\Lcal_2}_{\alpha} \otimes \mathsf{res}^{\Lcal_1,\Lcal_2}_{\beta})  \ , \\  \label{eq:proof2}
m_{\alpha,\beta}^{\Lcal_1,\Lcal_3} & = \mathsf{res}^{\Lcal_2, \Lcal_3}_{\alpha+\beta} \circ m_{\alpha,\beta}^{\Lcal_1,\Lcal_2}\ , 
\end{align}
as morphisms $A(\bHiggsaLone) \otimes A(\bHiggsbLone) \to A(\bHiggsabLthree)$. We will prove \eqref{eq:proof1} in details and leave the (easier) \eqref{eq:proof2} to the reader. We begin by observing that by Lemma \ref{lem:standard} there is a factorization
\begin{align}
\begin{aligned}
  \begin{tikzpicture}[xscale=3.5,yscale=-1.5]
  \node (A0_0) at (0, 0) {$\QaLoneLtwo\times \QbLoneLtwo$};
\node (A1_0) at (1, 0) {$\QtildeabLoneLthree$};
\node (A0_1) at (0, 1) {$\QaLone\times\QbLone$};
    \path (A0_0) edge [->]node [left] {$\scriptstyle{j}$} (A0_1);
   \path (A1_0) edge [->]node [above] {$\scriptstyle{\overline{q}_{\alpha, \beta}^{(\Lcal_1, \Lcal_2), \Lcal_3}}$} (A0_0);
   \path (A1_0) edge [->]node [auto] {$\scriptstyle{\overline{q}_{\alpha, \beta}^{\Lcal_1, \Lcal_3}}$} (A0_1);
  \end{tikzpicture}
\end{aligned}
\end{align}
where $j$ denote the canonical open immersion. Here for any pair of line bundles $\Lcal, \Lcal'$, by abuse of notation we have denoted by $\overline{q}_{\alpha,\beta}^{\Lcal, \Lcal'}$ the composition of the open embedding $\QtildeabLLp\hookrightarrow \QtildeabL$ with the map $\overline{q}_{\alpha,\beta}^{\Lcal, \Lcal'}$ introduced in \eqref{eq:overlinemaps}. 

Keeping the notations used in Section \ref{sec:localdescriptionHiggstilde}, set $G\coloneqq \mathsf{G}_{\alpha+\beta}^{\Lcal_1}$, $P\coloneqq \PabLone$ and $X_{\Lcal_2}\coloneqq G \underset{P}{\times} (\QaLoneLtwo\times \QbLoneLtwo)$. Then there is also a factorization
\begin{align}
\begin{aligned}
  \begin{tikzpicture}[xscale=2.5,yscale=-1]
  \node (A0_0) at (0, 0) {$T^\ast_G X$};
\node (A2_0) at (2, 0) {$Z^\circ_G$};
\node (A1_1) at (1, 1) {$T^\ast_G X_{\Lcal_2}$};
 \node (A0_2) at (0, 2) {$T^\ast X$};
\node (A2_2) at (2, 2) {$Z^\circ $};
\node (A1_3) at (1, 3) {$T^\ast X_{\Lcal_2}$};
    \path (A0_0) edge [->]node [above] {$\scriptstyle{}$} (A0_2);
    \path (A2_0) edge [->]node [above] {$\scriptstyle{}$} (A2_2);
   \path (A1_1) edge [->]node [above] {$\scriptstyle{}$} (A1_3);  
    \path (A2_0) edge [->]node [right=2em, above] {$\scriptstyle{\Phi_G}$} (A0_0);
   \path (A2_0) edge [->]node [auto] {$\scriptstyle{\Phi_G'}$} (A1_1);
    \path (A2_2) edge [->]node [right=2em, above] {$\scriptstyle{\Phi}$} (A0_2);
   \path (A2_2) edge [->]node [auto] {$\scriptstyle{\Phi'}$} (A1_3);
       \path (A1_1) edge [->]node [auto] {$\scriptstyle{j_G}$} (A0_0);
    \path (A1_3) edge [->]node [auto] {$\scriptstyle{j}$} (A0_2);
  \end{tikzpicture}
\end{aligned}
\end{align}

Because Gysin pullbacks commute with the restriction to open subsets, we have $\Phi^!=(\Phi')^! \circ j^\ast = (\Phi')^! \circ (\mathsf{res}^{\Lcal_1,\Lcal_2}_{\alpha} \otimes \mathsf{res}^{\Lcal_1,\Lcal_2}_{\beta})$. In order to conclude, we have to identify $(\Phi')^!$ with the Gysin pullback $\Phi_2^!$ coming from the cartesian square
\begin{align}
\begin{aligned}
  \begin{tikzpicture}[xscale=1.5,yscale=-1]
  \node (A0_0) at (0, 0) {$T^\ast_{G_2} X_2$};
\node (A2_0) at (2, 0) {$Z^\circ_{G_2}$};
 \node (A0_2) at (0, 2) {$T^\ast X_2$};
\node (A2_2) at (2, 2) {$Z^\circ_2 $};
    \path (A0_0) edge [->]node [above] {$\scriptstyle{}$} (A0_2);
    \path (A2_0) edge [->]node [above] {$\scriptstyle{}$} (A2_2);
    \path (A2_0) edge [->]node [above] {$\scriptstyle{\Phi_{2, G_2}}$} (A0_0);
    \path (A2_2) edge [->]node [above] {$\scriptstyle{\Phi_2}$} (A0_2);
  \end{tikzpicture}
\end{aligned}
\end{align}
itself built from the diagram
\begin{align}
\begin{aligned}
  \begin{tikzpicture}[xscale=2,yscale=-1]
  \node (A0_0) at (0, 0) {$X_2$};
\node (A1_0) at (1, 0) {$W_2^\circ$};
\node (A2_0) at (1.8, 0) {$X_2'$};
    \path (A1_0) edge [->]node [above] {$\scriptstyle{\overline{q}_2}$} (A0_0);
   \path (A1_0) edge [->]node [above] {$\scriptstyle{\overline{p}_2}$} (A2_0);
  \end{tikzpicture}
\end{aligned}
\end{align}
where we set
\begin{align}
G_2\coloneqq\mathsf{G}^{\Lcal_2}_{\alpha+\beta}\ ,\quad P_2\coloneqq \mathsf{P}_{\alpha,\beta}^{\Lcal_2}\ ,\quad X'_2\coloneqq {\mathsf{Q}}^{\Lcal_2,\Lcal_3}_{\alpha+\beta}\ , \quad W^{\circ}_2\coloneqq G_2 \underset{P_2}{\times} \widetilde{\mathsf{Q}}^{\Lcal_2,\Lcal_3}_{\alpha,\beta}\ , \quad X_2\coloneqq G_2 \underset{P_2}{\times}(\mathsf{Q}^{\Lcal_2}_{\alpha} \times \mathsf{Q}^{\Lcal_2}_{\beta})\ .
\end{align}
Consider the commutative diagram
\begin{align}\label{eq:diagram-proof1}
\begin{aligned}
  \begin{tikzpicture}[xscale=3.5,yscale=-2]
  \node (A0_0) at (0, 0) {$\mathsf{Q}^{\Lcal_2}_{\alpha} \times  \mathsf{Q}^{\Lcal_2}_{\beta}$};
\node (A1_0) at (1, 0) {$\widetilde{\mathsf{Q}}^{\Lcal_2,\Lcal_3}_{\alpha,\beta}$};
 \node (A2_0) at (2, 0) {$\mathsf{Q}^{\Lcal_2,\Lcal_3}_{\alpha+\beta}$};
 \node (A0_1) at (0, 1) {$\mathsf{R}^{\Lcal_1,\Lcal_2}_{\alpha} \times  \mathsf{R}^{\Lcal_1,\Lcal_2}_{\beta}$};
\node (A1_1) at (1, 1) {$\widetilde{\mathsf{R}}^{(\Lcal_1,\Lcal_2),\Lcal_3}_{\alpha,\beta}$};
 \node (A2_1) at (2, 1) {$\mathsf{R}^{(\Lcal_1,\Lcal_2),\Lcal_3}_{\alpha+\beta}$};
 \node (A0_2) at (0, 2) {$\mathsf{Q}^{\Lcal_1,\Lcal_2}_{\alpha} \times  \mathsf{Q}^{\Lcal_1,\Lcal_2}_{\beta}$};
\node (A1_2) at (1, 2) {$\widetilde{\mathsf{Q}}^{\Lcal_1,\Lcal_3}_{\alpha,\beta}$};
 \node (A2_2) at (2, 2) {$\mathsf{Q}^{\Lcal_1,\Lcal_3}_{\alpha+\beta}$};
    \path (A1_0) edge [->]node [above] {$\scriptstyle{}$} (A0_0);
    \path (A1_0) edge [->]node [above] {$\scriptstyle{}$} (A2_0);
    \path (A1_1) edge [->]node [above] {$\scriptstyle{}$} (A0_1);
    \path (A1_1) edge [->]node [above] {$\scriptstyle{}$} (A2_1);
    \path (A1_2) edge [->]node [above] {$\scriptstyle{}$} (A0_2);
    \path (A1_2) edge [->]node [above] {$\scriptstyle{}$} (A2_2);
    \path (A0_1) edge [->]node [above] {$\scriptstyle{}$} (A0_0);
    \path (A0_1) edge [->]node [above] {$\scriptstyle{}$} (A0_2);
    \path (A1_1) edge [->]node [above] {$\scriptstyle{}$} (A1_0);
    \path (A1_1) edge [->]node [above] {$\scriptstyle{}$} (A1_2);
    \path (A2_1) edge [->]node [above] {$\scriptstyle{}$} (A2_0);
    \path (A2_1) edge [->]node [above] {$\scriptstyle{}$} (A2_2);
  \end{tikzpicture}
\end{aligned}
\end{align}
where ${\mathsf{R}}^{(\Lcal_1,\Lcal_2),\Lcal_3}_{\alpha,\beta}$ is  the scheme representing the contravariant functor $\mathsf{Aff}/k \to (\mathsf{Sets})$ which assigns to an affine $k$-variety $S$ the set of pairs $([\phi], u)$, where $[\phi]\in \mathsf{Q}^{\Lcal_1,\Lcal_3}_{\alpha+\beta}(S)$, and $u\colon \VV_{\Lcal_2,\,\alpha+\beta}\boxtimes \Ocal_S\simeq \Hom(\Lcal_2\boxtimes\Ocal_S,\Fcal)$ is an isomorphism. 

The downwards pointing vertical arrows in \eqref{eq:diagram-proof1} are, respectively, a $\mathsf{G}^{\Lcal_2}_{\alpha} \times \mathsf{G}^{\Lcal_2}_{\beta}$, a $P_{\alpha, \beta}^{\Lcal_2}$ and a $\mathsf{G}^{\Lcal_2}_{\alpha+\beta}$-principal bundle, where $P_{\alpha, \beta}^{\Lcal_2}\subset \mathsf{G}^{\Lcal_2}_{\alpha+\beta}$ is a parabolic subgroup with Levi factor $\mathsf{G}^{\Lcal_2}_{\alpha} \times \mathsf{G}^{\Lcal_2}_{\beta}$. Similarly, the upwards pointing vertical arrows in \eqref{eq:diagram-proof1} are respectively a $\mathsf{G}^{\Lcal_1}_{\alpha} \times \mathsf{G}^{\Lcal_1}_{\beta}$, a $P_{\alpha,\beta}^{\Lcal_1}$ and a $\mathsf{G}^{\Lcal_1}_{\alpha+\beta}$-principal bundle, where $P_{\alpha,\beta}^{\Lcal_1} \subset \mathsf{G}^{\Lcal_1}_{\alpha+\beta}$ is a parabolic subgroup with Levi factor $\mathsf{G}^{\Lcal_1}_{\alpha} \times \mathsf{G}^{\Lcal_1}_{\beta}$. 

After passing to the cotangent spaces, the Gysin pullback $(\Phi')^!$ comes from the bottom row of \eqref{eq:diagram-proof1} while $\Phi_2^!$ comes from the
top row of \eqref{eq:diagram-proof1}. We are in the following general situation. Let $H \subset P \subset G$, $H'\subset P'\subset G'$ be a pair of Levi factors inclusions of parabolic subgroups of some reductive groups $G,G'$. Set $\overline{G}\coloneqq G \times G', \overline{P}=P \times P', \overline{H}=H \times H'$. Let $Y,V,X'$ be a triple of smooth varieties equipped with respective actions of $H,P$ and $G$ along with $P$-equivariant maps 
\begin{align}
\begin{aligned}
  \begin{tikzpicture}[xscale=1.5,yscale=-1]
  \node (A0_0) at (0, 0) {$Y$};
\node (A1_0) at (1, 0) {$V$};
\node (A2_0) at (2, 0) {$X'$};
    \path (A1_0) edge [->]node [above] {$\scriptstyle{q}$} (A0_0);
   \path (A1_0) edge [->]node [above] {$\scriptstyle{p}$} (A2_0);
  \end{tikzpicture}
\end{aligned}
\end{align}
We further assume that $q$ is smooth and $p$ is a closed embedding. Let $(\overline{Y},\overline{V}, \overline{X'}, \overline{q},\overline{p})$ be similar data for $\overline{H}, \overline{P}, \overline{G}$ and suppose that we have a commuting diagram
\begin{align}
\begin{aligned}
  \begin{tikzpicture}[xscale=1.5,yscale=-1]
  \node (A0_0) at (0, 0) {$\overline{Y}$};
\node (A1_0) at (1, 0) {$\overline{V}$};
\node (A2_0) at (2, 0) {$\overline{X'}$};
  \node (A0_1) at (0, 1) {$Y$};
\node (A1_1) at (1, 1) {$V$};
\node (A2_1) at (2, 1) {$X'$};
   \path (A1_0) edge [->]node [above] {$\scriptstyle{\overline{q}}$} (A0_0);
   \path (A1_0) edge [->]node [above] {$\scriptstyle{\overline{p}}$} (A2_0); 
    \path (A1_1) edge [->]node [above] {$\scriptstyle{q}$} (A0_1);
   \path (A1_1) edge [->]node [above] {$\scriptstyle{p}$} (A2_1);
   \path (A0_0) edge [->]node [above] {$\scriptstyle{}$} (A0_1);
   \path (A1_0) edge [->]node [above] {$\scriptstyle{}$} (A1_1);
      \path (A2_0) edge [->]node [above] {$\scriptstyle{}$} (A2_1); 
  \end{tikzpicture}
\end{aligned}
\end{align}
in which the vertical arrows are respectively a $H', P'$ and $G'$-principal bundles. Forming the fiber products 
\begin{align}
X\coloneqq G \underset{P}{\times}Y\ , \qquad W\coloneqq G \underset{P}{\times} V\ , \qquad \overline{X}\coloneqq \overline{G} \underset{\overline{P}}{\times}\overline{Y}\ , \qquad \overline{W}\coloneqq\overline{G} \underset{\overline{P}}{\times} \overline{V}
\end{align}
yields a commuting diagram
\begin{align}
\begin{aligned}
  \begin{tikzpicture}[xscale=1.5,yscale=-1]
  \node (A0_0) at (0, 0) {$\overline{X}$};
\node (A1_0) at (1, 0) {$\overline{W}$};
\node (A2_0) at (2, 0) {$\overline{X'}$};
  \node (A0_1) at (0, 1) {$X$};
\node (A1_1) at (1, 1) {$W$};
\node (A2_1) at (2, 1) {$X'$};
   \path (A1_0) edge [->]node [above] {$\scriptstyle{}$} (A0_0);
   \path (A1_0) edge [->]node [above] {$\scriptstyle{}$} (A2_0); 
    \path (A1_1) edge [->]node [above] {$\scriptstyle{}$} (A0_1);
   \path (A1_1) edge [->]node [above] {$\scriptstyle{}$} (A2_1);
   \path (A0_0) edge [->]node [above] {$\scriptstyle{}$} (A0_1);
   \path (A1_0) edge [->]node [above] {$\scriptstyle{}$} (A1_1);
      \path (A2_0) edge [->]node [above] {$\scriptstyle{}$} (A2_1); 
  \end{tikzpicture}
\end{aligned}
\end{align}
in which all vertical arrows are $G'$-principal bundles. Next, we put 
\begin{align}
Z\coloneqq T^\ast_W(X \times X')\ , \ \overline{Z}\coloneqq T^\ast_{\overline{W}}( \overline{X} \times \overline{X'})\ , \ Z_G\coloneqq Z \cap (T^\ast_GX \times T^\ast_GX')\ , \ \overline{Z}_{\overline{G}}\coloneqq \overline{Z} \cap (T^\ast_{\overline{G}}\overline{X} \times T^\ast_{\overline{G}}\overline{X'})\ .
\end{align}
Observe that $T^\ast_{\overline{G}}\overline{X}$ and $\overline{Z}_{\overline{G}}$ are both $G'$-principal bundles over $T^\ast_GX$ and $Z_G$ respectively. In addition, there is a commutative diagram
 \begin{align}
\begin{aligned}
  \begin{tikzpicture}[xscale=2.5,yscale=-1]
  \node (A0_0) at (0, 0) {$T^\ast_{\overline{G}}\overline{X}$};
\node (A2_0) at (2, 0) {$\overline{Z}_{\overline{G}}$};
  \node (A1_1) at (1, 1) {$T^\ast_G X$};
\node (A3_1) at (3, 1) {$Z_G$};
  \node (A0_2) at (0, 2) {$T^\ast_{{G}'}\overline{X}$};
\node (A2_2) at (2, 2) {$\overline{Z}_{G'}$};
  \node (A1_3) at (1, 3) {$T^\ast X$};
\node (A3_3) at (3, 3) {$Z$};
  \node (A0_4) at (0, 4) {$T^\ast\overline{X}$};
\node (A2_4) at (2, 4) {$\overline{Z}$};
   \path (A0_0) edge [->]node [above] {$\scriptstyle{}$} (A1_1);
   \path (A2_0) edge [->]node [above] {$\scriptstyle{}$} (A3_1); 
    \path (A0_2) edge [->]node [above] {$\scriptstyle{}$} (A1_3);
   \path (A2_0) edge [->]node [above] {$\scriptstyle{}$} (A3_1);
   \path (A2_2) edge [->]node [above] {$\scriptstyle{}$} (A3_3);
   \path (A2_0) edge [->]node [right=2em, above] {$\scriptstyle{\overline{\Phi}_{\overline{G}}}$} (A0_0);
      \path (A3_1) edge [->]node [right=2em, above] {$\scriptstyle{\Phi_{G}}$} (A1_1); 
      \path (A2_2) edge [->]node [right=2em, above] {$\scriptstyle{\overline{\Phi}_{G'}}$} (A0_2); 
      \path (A3_3) edge [->]node [right=2em, above] {$\scriptstyle{\Phi}$} (A1_3);
      \path (A2_4) edge [->]node [right=2em, above] {$\scriptstyle{\overline{\Phi}}$} (A0_4); 
   \path (A0_0) edge [->]node [above] {$\scriptstyle{}$} (A0_2);
   \path (A0_2) edge [->]node [above] {$\scriptstyle{}$} (A0_4);
   \path (A2_0) edge [->]node [above] {$\scriptstyle{}$} (A2_2);
   \path (A2_2) edge [->]node [above] {$\scriptstyle{}$} (A2_4);
   \path (A1_1) edge [->]node [above] {$\scriptstyle{}$} (A1_3);
   \path (A3_1) edge [->]node [above] {$\scriptstyle{}$} (A3_3);      
  \end{tikzpicture}
\end{aligned}
\end{align}
in which all diagonal arrows are $G'$-principal bundles. Note that $\overline{\Phi}_{G'}$ is regular since $\Phi$ is, hence by \cite[Theorem~6.2 (c)]{book:fulton1998} it follows that $\overline{\Phi}^!=\overline{\Phi}_{G'}^!$. Thanks to the identifications $A^{\overline{G}}(T^\ast_{\overline{G}}\overline{X}) \simeq A^G(T^\ast_GX)$ and $A^{\overline{G}}(\overline{Z}_{\overline{G}}) \simeq A^G(Z_G)$, we have that $\Phi^!=\overline{\Phi}^!$.

To conclude it is enough to apply the previous argument first for the pair of reductive groups $G=\mathsf{G}^{\Lcal_1}_{\alpha+\beta}, G'=\mathsf{G}^{\Lcal_2}_{\alpha+\beta}$ and the diagram in \eqref{eq:diagram-proof1} consisting of the central and bottom vertical arrows, later for the pair of reductive groups $G'=\mathsf{G}^{\Lcal_1}_{\alpha+\beta}, G=\mathsf{G}^{\Lcal_2}_{\alpha+\beta}$ and the diagram in \eqref{eq:diagram-proof1} consisting of the central and top vertical arrows. We obtain $\Phi_2^!= \overline{\Phi}=(\Phi')^!$ and this completes the proof of \eqref{eq:proof1}. 

The proof of associativity of the multiplication can be made in locally and uses the same arguments as in the proof of \cite[Theorem~2.2]{art:minets2018}.
\end{proof}

\begin{definition} Let $A$ be either a free OBM theory or $A=H_\ast$. We call the algebra $\COHA'_{\Higgs(X)}$ the \emph{unrestricted} $A$-\emph{homological Hall algebra of the category $\Higgs(X)$}. 
\end{definition}
It is easy to see from the construction that the subgroup 
\begin{align}
\COHA_{\Higgs(X)}\coloneqq \bigoplus_{\alpha \in (\Z^2)^+} A^0(\bHiggsa)
\end{align}
is a subalgebra, which we call the $A$-\emph{homological Hall algebra of the category $\Higgs(X)$}.

\subsection{Cohomological Hall algebra of nilpotent, semistable and equivariant Higgs sheaves}

The full abelian subcategory $\Higgs^{\mathsf{nilp}}(X)$ of $\Higgs(X)$ is stable under extension, and the same holds for $\Higgs^{\mathsf{ss},\, \nu}(X)$ for each fixed slope $\nu$. We may repeat the above constructions \emph{verbatim} in these contexts (using refined Gysin pullback maps on each local charts\footnote{The stack of semistable Higgs bundles is a global quotient stack (see for example \cite[Section~7.7.1]{art:casalainawise2017}), hence we do not need to restrict ourselves to local charts.} obtained by embedding in the \emph{same} l.c.i. morphism). The compatibility of Gysin pullbacks with direct images by proper morphism and and pullback by open immersions (see \cite[Theorem~6.2 (a), (b)]{book:fulton1998}) imply the following.
\begin{corollary}\label{cor:Higgsvariantnilp} 
There is a natural associative algebra structure on the group 
\begin{align}
\COHA'_{\Higgs^{\mathsf{nilp}}(X)}\coloneqq\bigoplus_{\alpha \in (\Z^2)^+} A(\bLama)\ .
\end{align}
Furthermore, the proper pushforward map $\COHA'_{\Higgs^{\mathsf{nilp}}(X)} \to \COHA'_{\Higgs(X)}$ is an algebra homomorphism. Similarly, there is a natural associative algebra structure on the group 
\begin{align}
\COHA_{\Higgs^{\mathsf{nilp}}(X)}\coloneqq\bigoplus_{\alpha \in (\Z^2)^+} A^0(\bLama)\ ,
\end{align}
and the proper pushforward map $\COHA_{\Higgs^{\mathsf{nilp}}(X)} \to \COHA_{\Higgs(X)}$ is an algebra homomorphism.
\end{corollary}

\begin{corollary}\label{cor:Higgsvariantss} 
For any fixed slope $\nu$, there is a natural associative algebra structure on the group 
\begin{align}
\COHA_{\Higgs^{\mathsf{ss},\, \nu}(X)}\coloneqq\bigoplus_{\genfrac{}{}{0pt}{}{\alpha \in (\Z^2)^+}{\mu(\alpha)=\nu}} A^0(\bHiggs^{\mathsf{ss}}_{\alpha})\ .
\end{align}
Furthermore, the open restriction map $\COHA_{\Higgs^\nu(X)} \to \COHA_{\Higgs^{\mathsf{ss},\, \nu}(X)}$ is an algebra homomorphism, where
\begin{align}
\COHA_{\Higgs^{\nu}(X)}\coloneqq\bigoplus_{\substack{ \alpha \in (\Z^2)^+ \\ \mu(\alpha)=\nu}} A^0(\bHiggs_{\alpha})\ .
\end{align}
\end{corollary}

One may likewise consider equivariant versions of all the above, with respect to the action of the multiplcative group $T=\mathbb{G}_m$ on $\bHiggsa$ by $z \cdot (\Fcal,\theta)\coloneqq(\Fcal,z\,\theta)$, and get in this fashion \emph{equivariant} $A$-homological Hall algebras
\begin{align}
\COHA'^{\,T}_{\Higgs(X)}\ , \qquad \COHA'^{\, T}_{\Higgs^{\mathsf{nilp}}(X)}\ , \quad \COHA^T_{\Higgs(X)}\ , \quad \COHA^T_{\Higgs^{\mathsf{nilp}}(X)}\ , \quad \COHA^T_{\Higgs^{\mathsf{ss},\, \nu}(X)}\ .
\end{align}
These are, by construction, modules over the ring $A_T^\ast(\pt) \simeq A^\ast(\pt)[[c_1(t)]]$, where $c_1(t)$ is the first Chern class of the tautological character of $T$, see e.g. \cite[Appendix~A]{art:minets2018}. 

\begin{proposition}\label{prop:localizationT} 
The direct image morphism is an isomorphism of \emph{localized} algebras
\begin{align}
\COHA^T_{\Higgs^{\mathsf{nilp}}(X)} \otimes_{A_T(\pt)} \Frac(A_T(\pt)) \stackrel{\sim}{\to} \COHA^T_{\Higgs(X)} \otimes_{A_T(\pt)} Frac(A_T(\pt))\ .
\end{align}
\end{proposition}
\begin{proof} 
Fix a line bundle $\Lcal$ and $\alpha \in (\Z^2)^+$. Let $(T^\ast_{\GaL}\QaL)^{\mathsf{nilp}} \subset T^\ast_{\GaL}\QaL$ be the closed $\GaL$-subvariety such that
\begin{align}
\bLama^{>\Lcal}\coloneqq\big[(T^\ast_{\GaL}\QaL)^{\mathsf{nilp}}  / \GaL\big]\simeq \bHiggsa^{>\Lcal} \underset{\bHiggsa}{\times} \bLama\ .
\end{align}
It is enough to show that the direct image morphism
\begin{align}
A^{T \times \GaL}((T^\ast_{\GaL}\QaL)^{\mathsf{nilp}})\otimes_{A_T(\pt)} \Frac(A_T(\pt)) \to A_{T \times \GaL}(T^\ast_{\GaL}\QaL)\otimes_{A_T(\pt)} \Frac(A^T(\pt))
\end{align}
is an isomorphism. This follows from the same argument as in \cite[Corollary~6.3]{art:minets2018}.
\end{proof}

The above proposition allows one to deduce certain properties of $\COHA^T_{\Higgs(X)}$ from the geometry of $\displaystyle\bLam\coloneqq\bigsqcup_{\alpha} \bLama$, which is sometimes more agreable than that of $\displaystyle\bHiggs\coloneqq\bigsqcup_{\alpha} \bHiggsa$. The results in the next two sections provide an illustration of this principle.

\bigskip\section{Torsion-freeness}\label{sec:torsionfreeness}

In this section and in the next, we take $A=H_\ast$, the usual Borel-Moore homology with rational coefficients (most results also hold for Chow groups with appropriate modifications). In addition, we will focus on the cohomological Hall algebras $\COHA^T_{\Higgs^{\mathsf{nilp}}(X)}$ and $\COHA_{\Higgs^{\mathsf{nilp}}(X)}$, which we denote simply by $\COHA^T_{\bLam}$ and $\COHA_{\bLam}$.
 
\subsection{The universal cohomology ring of $\bCoh$} 

For any $\alpha \in (\Z^2)^+$ there is an action of $H^\ast(\bCoha)$ on $H_\ast(\bHiggsa)$ defined by $c \cdot u\coloneqq r_{\alpha}^\ast(c) \cap u$, where $r_{\alpha}\colon \bHiggsa \to \bCoha$ is the canonical projection. The ring $H^\ast(\bCoha)$ being freely generated by tautological classes (see Theorem \ref{T:Heinloth}), it is independent of $\alpha$ (strictly speaking, this is true only for $\rk(\alpha) >0$). In this paragraph, we consider a universal version $\HH$ of this ring, and endow the algebras $\COHA_{\bLam}, \COHA_{\Higgs(X)}$ and $\COHA_{\Higgs^{\mathsf{ss},\, \nu}(X)}$ with $\HH$-module algebra structures.

Recall our fixed  basis $\Pi=\{1, \pi_1, \ldots, \pi_{2g}, \varpi\}$ of $H^\ast(X)$, with $1 \in H^0(X), \pi_1, \ldots, \pi_{2g} \in H^1(X)$ and $\varpi \in H^2(X)$. Let $\HH\coloneqq\Q[c_{i,\pi}]_{i,\pi}$ be the graded free supercommutative algebra generated by elements $c_{i,\pi}$ with $i \geq 1, \pi \in \Pi$. The degree of $c_{i,\pi}$ is defined to be $\deg(c_{i,\pi})=2(i-1)+\deg(\pi)$. Note that we include the degree zero element $c_{1,1}$. For any $\alpha$, there is a surjective morphism $a_\alpha\colon \HH \to H^\ast(\bCoha)$ defined by $a_\alpha(c_{i,\pi})=c_{i,\pi}(\Efrak_\alpha)$, where the classes $c_{i,\pi}(\Efrak_\alpha)$ are defined in \eqref{eq:chernclasses}. The kernel of the map $a_\alpha$ is the ideal generated by the element $c_{1,1}-deg(\alpha)$ whenever $\rk(\alpha)>0$. Via the map $a_\alpha$, the ring $\HH$ acts on $\COHA_{\Higgs(X)}$, preserving the class $\alpha$ (but shifting the cohomological degree). This action factors through to an action on $\COHA_{\Higgs^{\mathsf{ss},\, \nu}(X)}$, and there is a compatible action on $\COHA_{\bLam}$.

Let $\alpha_1, \alpha_2 \in (\Z^2)^+$ and set $\alpha=\sum_i \alpha_i$. Define a morphism
\begin{align}
\Delta_{\alpha_1,\, \alpha_2}\colon H^\ast(\bCoha) \to H^\ast(\bCoh_{\alpha_1}) \otimes  H^\ast(\bCoh_{\alpha_2})
\end{align}
as the pullback by the direct sum morphism 
\begin{align}
\bigoplus_{\alpha_1,\, \alpha_2}\colon \bCoh_{\alpha_1} \times \bCoh_{\alpha_2} \to \bCoh_{\alpha}\ ,\quad (\Fcal_1, \Fcal_2) \mapsto \Fcal_1 \oplus \Fcal_2\ .
\end{align}

\begin{lemma} 
There exists a (unique) coassociative coproduct $\Delta\colon \HH \to \HH \otimes \HH$ such that, for any $\alpha_1, \alpha_2$ we have
\begin{align}
(a_{\alpha_1} \otimes a_{\alpha_2} ) \circ \Delta= \Delta_{\alpha_1,\, \alpha_2} \circ a_{\alpha_1+\alpha_2} \ . 
\end{align}
Equipped with this coproduct, $\HH$ is a graded commutative and cocommutative Hopf algebra.
\end{lemma}
\begin{proof} 
The map $\Delta_{\alpha_1,\alpha_2}$ is nicely compatible with tautological classes as 
\begin{align}
{\bigoplus_{\alpha_1,\, \alpha_2}}^{\!\!\ast}(\Efrak_\alpha)\simeq \mathsf{pr}_1^\ast(\Efrak_{\alpha_1}) \oplus
\mathsf{pr}_2^\ast(\Efrak_{\alpha_2})\ ,
\end{align}
where $\mathsf{pr}_i\colon \bCoh_{\alpha_1} \times \bCoh_{\alpha_2} \to \bCoh_{\alpha_i}$ is the projection for $i=1, 2$. 
In particular, 
\begin{align}
\Delta_{\alpha_1,\, \alpha_2}(c_{l}(\Efrak_\alpha))=\sum_{i + j =l}c_{i}(\Efrak_{\alpha_1}) \otimes_{H^*(X)} c_{j}(\Efrak_{\alpha_2})\ ,
\end{align}
where we use the convention that $c_0(\Efrak)=1$. Taking K\"unneth components yields
\begin{align}
\Delta_{\alpha_1,\, \alpha_2}(c_{l,\varpi}(\Efrak_\alpha)) &=\sum_{i + j =l}c_{i,\varpi}(\Efrak_{\alpha_1}) \otimes c_{j,\varpi}(\Efrak_{\alpha_2})\ ,\\[2pt]
\Delta_{\alpha_1,\, \alpha_2}(c_{l,\pi_u}(\Efrak_\alpha)) &=\sum_{i + j =l} \left(c_{i,\pi_u}(\Efrak_{\alpha_1}) \otimes c_{j,\varpi}(\Efrak_{\alpha_2}) + c_{i,\varpi}(\Efrak_{\alpha_1}) \otimes c_{j,\pi_u}(\Efrak_{\alpha_2})\right)\ ,\\[2pt]
\Delta_{\alpha_1,\, \alpha_2}(c_{l,1}(\Efrak_\alpha))& =\sum_{i + j =l}\big(c_{i,1}(\Efrak_{\alpha_1}) \otimes c_{j,\varpi}(\Efrak_{\alpha_2}) + c_{i,\varpi}(\Efrak_{\alpha_1}) \otimes c_{j,1}(\Efrak_{\alpha_2}) + \sum_u c_{i,\pi_u}(\Efrak_{\alpha_1}) \otimes c_{j,\pi_u^*}(\Efrak_{\alpha_2})\big)\ .
\end{align}
We formally define the comultiplication $\Delta$ on $\HH$ using the above formulas. The fact that it is an algebra morphism follows from the compatibility of pull-backs with cup products.
\end{proof}

We will say that an $\HH$-module $\mathbf{M}$ is an $\HH$-module algebra if $\mathbf{M}$ is equipped with an associative algebra structure $m\colon \mathbf{M} \otimes \mathbf{M} \to \mathbf{M}$ such that for any $h \in \HH$, $x,y \in \mathbf{M}$ it holds $h \cdot m(x\otimes y)=m(\Delta(h) \cdot (x \otimes y))$.

\begin{proposition}\label{prop:Hmodule} 
The algebras $\COHA_{\Higgs(X)}, \COHA_{\bLam}$ and $\COHA_{\Higgs^{\mathsf{ss},\, \nu}(X)}$ are $\HH$-module algebras.
\end{proposition}
\begin{proof} 
Fix some $\alpha_1, \alpha_2$ and set $\alpha=\alpha_1+\alpha_2$. Let $\gamma \in H^\ast(\bCoha)$ and $u_i \in H^\ast(\bHiggs_{\alpha_i})$ for $i=1,2$.  Keep the notation of Section \ref{sec:COHAcoh}. By definition, $u_1 \cdot u_2= (p_{\alpha_1,\, \alpha_2})_\ast\Phi^!(u_1 \otimes u_2)$. Set $v\coloneqq\Phi^!(u_1 \otimes u_2)$. By the projection formula, we have 
\begin{align}
\gamma \cdot (p_{\alpha_1,\alpha_2})_\ast(v)= (p_{\alpha_1,\,\alpha_2})_\ast((p_{\alpha_1,\,\alpha_2})^\ast r_{\alpha}^\ast(\gamma) \cap v)\ .
\end{align}
Note that 
\begin{align}
 (p_{\alpha_1,\,\alpha_2})^\ast r_{\alpha}^\ast(\gamma)=\Phi_G^\ast (r_{\alpha_1}^\ast \otimes r_{\alpha_2}^\ast) (\Delta_{\alpha_1,\,\alpha_2}(\gamma))\ ,
\end{align}
while by multiplicativity of the Gysin pullback
\begin{align}
\Phi_G^\ast (r_{\alpha_1}^\ast \otimes r_{\alpha_2}^\ast) (\Delta_{\alpha_1,\, \alpha_2}(\gamma)) \cap \Phi^!(u_1 \otimes u_2)=\Phi^!((r_{\alpha_1}^\ast \otimes r_{\alpha_2}^\ast) (\Delta_{\alpha_1,\, \alpha_2}(\gamma)) \cap (u_1 \otimes u_2))\ .
\end{align}
This yields the desired equality. We are done.
\end{proof}

Proposition \ref{prop:Hmodule} has an obvious equivariant avatar. Note that in that case, $\HH$ is replaced by $\HH\otimes \Q[t]$.

\subsection{Torsion-freeness}

In the context of quivers, see \cite[Section~4.4]{art:schiffmannvasserot2017} (or rank zero Higgs sheaves, see \cite[Section~6]{art:minets2018}) the following technical result is crucial in describing the cohomological Hall algebras as \emph{shuffle algebras}. Although we do not give such a realization here, we nevertheless state the following theorem. 
\begin{theorem}\label{T:torsionfree} 
Let $\alpha\in (\Z^2)^+$. Then $H^{0, T}_\ast(\bLama)$ is a torsion-free $H^\ast(\bCoha)\otimes \Q[t]$-module.
\end{theorem}
\begin{proof}
Our approach will bear some similarity with those followed in the proofs of \cite[Proposition~4.6]{art:schiffmannvasserot2017} and \cite[Theorem~6.4]{art:minets2018}. To simplify the notation, we shall drop $0$ from $H^{0,  T}_\ast$. 

Let $\HH^+ \subset H^\ast(\bCoha)$ be the graded augmentation ideal of $H^\ast(\bCoha)$ and set $I=\HH^+ \otimes \Q[t] \subset H^\ast(\bCoha)\otimes \Q[t]$. For any $H^\ast(\bCoha)\otimes \Q[t]$-module $M$ we denote by $M_{\mathsf{loc},I}$ the localization of $M$ with respect to the
ideal $I$. We will prove the following two statements:
\begin{enumerate}\itemsep0.2cm
\item[(a)] the natural map $H^T_\ast(\bLama) \to H^T_\ast(\bLama)_{\mathsf{loc},I}$ is injective, i.e $H^T_\ast(\bLama)$ is $I$-torsion-free,
\item[(b)] the direct image morphism $H^T_\ast(\bCoha)_{\mathsf{loc},I} \to H^T_\ast(\bLama)_{\mathsf{loc},I}$ is an isomorphism.
\end{enumerate}
The theorem will follow since $H_\ast^T(\bCoha)$ is evidently torsion free (in fact, free) as a $H^\ast(\bCoha)\otimes \Q[t]$-module (note that $T$ acts trivially on $\bCoha$).

We begin with (a). Recall that we are only considering admissible classes in $H^{T}_\ast(\bLama)$, i.e., classes supported on finitely many irreducible components. Moreover, (see \eqref{eq:filtration} and \eqref{eq:gradedcoh}), there is a $H^\ast(\bCoha)\otimes \Q[t]$-invariant filtration of $H^T_\ast(\bLama)$ by $H^T_\ast(\bLam_{\preceq \ual})$, whose associated graded is 
\begin{align}
\mathsf{gr}(H^{T}_\ast(\bLama))=\bigoplus_{\ual \in J_{\alpha}} H^{T}_\ast(\bLam_{\ual})\ , \quad H^{T}_\ast(\bLam_{\ual}) \simeq \bigotimes_i\, H^\ast_T(\bCoh_{\alpha_i})\ .
\end{align}
In particular, each $H^T_\ast(\bLam_{\ual})$ is a \emph{free} graded $\Q[t]$-module, and hence also $I$-torsion-free. Therefore $H^T_\ast(\bLama)$ is also $I$-torsion-free.

We will check statement (b) locally, by showing that for any line bundle $\Lcal$, the direct image morphism 
\begin{align}\label{eq:prooftortwo}
H^T_\ast(\bCohaL)_{\mathsf{loc},I} \to H^T_\ast(\bLama^{>\Lcal})_{\mathsf{loc},I}
\end{align}
which respect to the closed embedding $\bCohaL\simeq \bLam_{(\alpha)}^{>\Lcal} \hookrightarrow \bLama^{\Lcal}$, is an isomorphism. The isomorphism is preserved by taking the limit with respect to $\Lcal$, indeed we may again consider the $H^\ast(\bCoha)\otimes \Q[t]$-invariant filtration as in case (a) above and argue on each $\bLam_{\preceq \ual}$ (which is an admissible stack). So let us fix a line bundle $\Lcal$. The open substack $\bCohaL$ is isomorphic to the quotient $[\QaL/\GaL]$, while $\bLama^{>\Lcal}=[(T^\ast_{\GaL}\QaL)^{\mathsf{nilp}}  / \GaL]$. Hence $H_\ast(\bCohaL) \simeq H_\ast^{\GaL}(\QaL)$ and $H_\ast(\bLama^{>\Lcal})\simeq H_\ast^{\GaL}((T^\ast_{\GaL}\QaL)^{\mathsf{nilp}})$. In particular, $H_\ast(\bCohaL)$ carries an action of the equivariant cohomology ring $H^\ast_{\GaL}(\pt)\eqqcolon \mathbf{R}_{\GaL}$. There is a similar action of $\mathbf{R}_{\GaL} $ on $H_\ast(\bHiggsaL)$ and on $H_\ast(\bLama^{>\Lcal})$.

We shall need the following result.
\begin{lemma}\label{L:prooftor1} 
For any $\alpha, \Lcal$ there is a surjective algebra morphism $s_{\Lcal,\alpha} \colon H^\ast(\bCoha) \to \mathbf{R}_{\GaL}$ such that for any $\gamma \in  H^\ast(\bCoha)$ and any $c \in H_\ast(\bHiggsa)$ (resp.\ $H_\ast(\bLama)$) we have
\begin{align}
(\gamma \cap c)\vert_{\bHiggsaL}& =s_{\Lcal,\alpha}(\gamma) \cap c\vert_{\bHiggsaL}\ .\\[4pt]
\text{(resp.\ }\ (\gamma \cap c)\vert_{\bLama^{>\Lcal}}& =s_{\Lcal,\alpha}(\gamma) \cap c\vert_{\bLama^{>\Lcal}}\ .\text{ )}
\end{align}
\end{lemma}
\begin{proof} 
Since the $H^\ast(\bCoha)$-action is given via pullback with respect to the morphism $r_\alpha\colon \bHiggsa\to \bCoha$, it suffices to prove the statement of the lemma for $\bCoha$ in place of $\bHiggsa$ or $\bLama$. 

Let $p_\alpha^\Lcal\colon \QaL \times X \to \QaL$ and $p_X\colon \QaL\times X\to X$ denote the two projections.  The action of $\mathbf{R}_{\GaL}$ on $H^{\GaL}_\ast(\QaL)$ is given by cap product by the Chern classes of the tautological $\GaL$-equivariant vector bundle $\VV\coloneqq\mathbb{R}(p_\alpha^\Lcal)_\ast(p_X^\ast\Lcal \otimes \Efrak_\alpha^\Lcal)$, whose fiber over a point $\big[\phi\colon \Lcal\otimes k^{\, \langle \overline{\Lcal},\alpha\rangle}\twoheadrightarrow \Fcal\big]$ is $\Hom(\Lcal, \Fcal)$.  On the other hand, the action of $H^\ast(\bCoha)$ is given by cap product with the K\"unneth components of the Chern classes of the tautological sheaf $\Efrak_\alpha$. Applying the Grothendieck-Riemann-Roch formula \cite[Theorem~15.2]{book:fulton1998} to the morphism $p_\alpha^\Lcal$ and $p_X^\ast\Lcal^\vee \otimes \Efrak_\alpha^\Lcal$ yields an expression for $\mathsf{ch}(\VV)$ in terms of $\mathsf{ch}(\Efrak_\alpha^\Lcal)$ as wanted. Note that $(p_\alpha^\Lcal)_\ast p_X^\ast \Lcal$ is a trivial $\GaL$-equivariant bundle on $\QaL$.
\end{proof}

Put $I_{\GaL}\coloneqq\mathbf{R}_{\GaL}^+ \otimes \Q[t]$, where $\mathbf{R}_{\GaL}^+$ is the graded augmentation ideal of $\mathbf{R}_{\GaL}$. By the relative form of the localization theorem (see e.g. \cite[Proposition~A.13, Theorem~A.14]{art:minets2018}) the localized pushforward map 
\begin{align}
H_\ast^{\GaL \times T}(\QaL)_{\mathsf{loc},I_{\GaL}} = H_\ast^{\GaL \times T}((T^\ast_{\GaL}\QaL)^T)_{\mathsf{loc},I_{\GaL}} \to H_\ast^{\GaL \times T}((T^\ast_{\GaL}\QaL)^{\mathsf{nilp}})_{\mathsf{loc}, I_{\GaL}}
\end{align}
is an isomorphism. By Lemma \ref{L:prooftor1} we have $s_{\Lcal,\alpha}(I)=I_{\GaL}$. This implies \eqref{eq:prooftortwo} and concludes the proof Theorem \ref{T:torsionfree}.
\end{proof}

\bigskip\section{Generation theorem}\label{sec:generation}

In this section, we take again $A=H_\ast$. For any $\alpha \in (\Z^2)^+$ there is a distinguished irreducible component $\bLam_{(\alpha)}$ of $\bLama$, namely the zero section of the projection $r_{\alpha}\colon \bHiggsa \to \bCoha$. Thus $\bLam_{(\alpha)} \simeq \bCoha$. In particular, by Poincar\'e duality,
\begin{align}
H_*(\bLam_{(\alpha)}) \simeq \left\{
\begin{array}{ll}
\Q[c_{i,\pi}(\Efrak_{\alpha}]_{i,\pi} & \text{if }\rk(\alpha)>0\ ,\\[4pt]
S^d(H^\ast(X)[z]) & \text{if } \alpha=(0,d)\ .
\end{array}
\right.
\end{align}
The following is an analog of \cite[Theorem~B (e)]{art:schiffmannvasserot2017}. It is interesting that the proof, though similar, is simpler in the curve case than in the quiver case as the structure of $\bLama$ is simpler than that of the Lusztig nilpotent stack.

\begin{theorem}\label{T:gen} 
For $A=H_\ast$, the algebra $\COHA_{\bLam}$ is generated by the collection of subspaces $H_\ast(\bLam_{(\alpha)})$ for $\alpha \in (\Z^2)^+$.
\end{theorem}
\begin{proof} 
Let us denote by $\mathbf{B}$ the subalgebra of $\COHA_{\bLam}$ generated by the collection of subspaces $H_\ast(\bLam_{(\alpha)})$ for $\alpha \in (\Z^2)^+$. By definition of $A_\ast^0$, every class in $c \in H_\ast^0(\bLama)$ is supported on a finite number of irreducible components, i.e., there exists a finite subset of Jordan types $I_c \subset J_{\alpha}$ such that 
\begin{align}
c \in \mathsf{Im}\bigg( H_\ast\Big(\bigsqcup_{\ual \in I_c} \bLam_{\ual}\Big) \to H^0_\ast(\bLama)\bigg) \ .
\end{align}
Recall the partial order $\prec$ on $J_{\alpha}$ (see \eqref{eq:preceq} and Proposition \ref{prop:preceq}) as well as the induced filtration \eqref{eq:filtration} on $H^0_\ast(\bLama)$. We will prove by induction on $\ual$ with respect to $\prec$ that 
\begin{align}\label{eq:proofgen1}
\mathsf{Im}\big( H_\ast(\bLam_{\preceq \ual}) \to H^0_\ast(\bLama)\big) \subset \mathbf{B}\ .
\end{align}
So let us fix $\ual \in J_{\alpha}$ and assume that \eqref{eq:proofgen1} holds for all $\underline{\beta}$ with $\underline{\beta} \prec \ual$. If $\ual=(\alpha)$ then \eqref{eq:proofgen1} holds by definition. Otherwise, let us write $\ual=(\alpha_1, \ldots, \alpha_s)$ and put $\gamma_i\coloneqq\sum_{j \geq i} \alpha_i((i-j)\deg(\omega_X))$; this is the total class of the $i$th row of the colored Young diagram associated with $\ual$, see \eqref{diag:Young}. Consider the (iterated) convolution diagram for Higgs stacks
\begin{align}\label{eq:proofgen2}
\begin{aligned}
  \begin{tikzpicture}[xscale=4,yscale=-1]
  \node (A0_0) at (0, 0) {$\prod_i \bHiggs_{\gamma_i}$};
\node (A1_0) at (1, 0) {$\widetilde{\bHiggs}_{\gamma_1, \ldots, \gamma_s}$};
\node (A2_0) at (1.8, 0) {$\bHiggs_\alpha$};
    \path (A1_0) edge [->]node [above] {$\scriptstyle{q_{\underline{\gamma}}}$} (A0_0);
   \path (A1_0) edge [->]node [above] {$\scriptstyle{p_{\underline{\gamma}}}$} (A2_0);
  \end{tikzpicture}
\end{aligned}
\end{align}

Using \eqref{eq:eulerformhiggs} and Theore~\ref{T:StackHiggs} (a) we compute
\begin{align}\label{E:proofgen3}
\dim(q_{\underline{\gamma}})=\dim(\widetilde{\bHiggs}_{\gamma_1, \ldots, \gamma_s})-\sum_i \dim( \bHiggs_{\gamma_i})=-2\sum_{i \neq j} \langle \gamma_i,\gamma_j\rangle\ .
\end{align}
We will use the following three observations:
\begin{enumerate}\itemsep0.2em
\item[(a)] $p_{\underline{\gamma}} \circ q_{\underline{\gamma}}^{-1}(\prod_i \bLam_{(\gamma_i)}) \subseteq \bLam_{\preceq \ual}$,
\item[(b)] $p_{\underline{\gamma}}\colon p_{\underline{\gamma}}^{-1}(\bLam_{\ual}) \to \bLam_{\ual}$ is an isomorphism,
\item[(c)] there exists an open subset $\Uscr$ of $\prod_i \bLam_{(\gamma_i)}$ over which $q_{\underline{\gamma}}$ is smooth with connected fibers of dimension $-\sum_{i\neq j}\langle \gamma_i,\gamma_j\rangle$ and $q_{\underline{\gamma}}^{-1}(\Uscr) \underset{\widetilde{\bHiggs}_{\gamma_1, \ldots, \gamma_s}}{\times}  p_{\underline{\gamma}}^{-1}(\bLam_{\ual})$ is open in $ p_{\underline{\gamma}}^{-1}(\bLam_{\ual})$.
\end{enumerate}
Statement (a) is easy, while (b) comes from the unicity of the iterated kernel filtration $\ker(\theta) \subseteq \ker(\theta^2) \subseteq \cdots \subseteq \Fcal$ for any $\underline{\Fcal} \in \bLam_{\ual}$. Statement (c) is proved as \cite[Proposition~1.6, Theorem~1.4]{art:bozec2016}, see also \cite[Lemma~3.19]{art:schiffmannvasserot2017}. Note that
\begin{align}
\dim(\bLam_{(\alpha)})-\sum_i \dim(\bLam_{(\gamma_i)})=-\sum_{i \neq j} \langle \gamma_i,\gamma_j\rangle\ .
\end{align}
From (a), (b) and (c), using the local construction of the multiplication map, the base change property of refined Gysin pullbacks \cite[Theorem 6.2 (b)]{book:fulton1998} we deduce that
\begin{align}\label{eq:proofgen4}
\Big( [\bLam_{(\gamma_s)}] \star[\bLam_{(\gamma_{s-1})}] \star \cdots \star [\bLam_{(\gamma_1)}]\Big)\vert_{\bLam_{\ual}}=[\bLam_{\ual}]
\end{align}
while $\mathsf{supp}\Big( [\bLam_{(\gamma_s)}] \star [\bLam_{(\gamma_{s-1})}] \star \cdots \star [\bLam_{(\gamma_1)}]\Big) \subseteq \bLam_{\preceq \ual}$. More generally, from the compatibility of refined Gysin morphisms with respect to cap product with Chern classes \cite[Proposition~6.3]{book:fulton1998}, we deduce that for any polynomials $P_1, \ldots, P_s$ in the (K\"unneth components of the) Chern classes of the tautological sheaves $\Efrak_{\gamma_1}, \ldots, \Efrak_{\gamma_s}$ on $\bLam_{(\gamma_1)} \simeq \bCoh_{\gamma_1}, \ldots, \bLam_{(\gamma_s)} \simeq \bCoh_{\gamma_s}$ respectively, we have
\begin{multline}\label{eq:proofgen5}
\Big(P_s(c_{i,\pi}(\Efrak_{\gamma_s}) \cap [\bLam_{(\gamma_s)}])\Big) \star\cdots \star \Big(P_1(c_{i,\pi}(\mathcal{E}_{\gamma_1}) \cap [\bLam_{(\gamma_1)}])\Big)\vert_{\bLam_{\ual}}\\
=\Big(P_1(c_{i,\pi}(\Efrak_1)) \cdots P_s(c_{i,\pi}(\Efrak_{s})\Big) \cap [\bLam_{\ual}]\ ,
\end{multline}
where $\Efrak_1, \ldots, \Efrak_s$ are the tautological sheaves over $\bLam_{\ual} \times X$ defined as $\Efrak_i = \ker(\theta^i)/\ker(\theta^{i-1})$, where $\theta$ is the Higgs field on the universal sheaf $\Efrak_\alpha$. We claim that $H^\ast(\bLam_{\ual})$ is generated by the Chern classes $c_{i,\pi}(\Efrak_j)$ for $j=1, \ldots, s$. Indeed, by Proposition \ref{prop:MS5.2}, $H^\ast(\bLam_{\ual})$ is generated by the Chern classes $c_{i,\pi}(\Efrak_{\alpha_j})$ for $j=1, \ldots, s$, where $[\Efrak_{\alpha_j}]=[\Efrak_{i+1}\otimes \omega_X]-[\Efrak_i]$ in $K_0(\bLam_{\ual})$. Since the K\"unneth components $c_{i,\pi}(\Efrak_{i+1}\otimes \omega_X)$ obviously generate the same algebra as the K\"unneth components $c_{i,\pi}(\Efrak_{i+1})$, the claim follows.
From all this, we deduce that
\begin{align}
H_\ast(\bLam_{\ual}) \subseteq \mathsf{gr}\big( H_\ast(\bLam_{(\gamma_s)}) \star \cdots \star  H_\ast(\bLam_{(\gamma_1)})\big) \subseteq \bigoplus_{\underline{\beta} \preceq \ual} H_\ast(\bLam_{\underline{\beta}})\ .
\end{align}
Using the induction hypothesis \eqref{eq:proofgen1} we get that
\begin{align}
\mathsf{Im}\big( H_\ast(\bLam_{\preceq \ual}) \to H^0_\ast(\bLama)\big) \subset \mathbf{B}
\end{align}
as wanted. Theorem \ref{T:gen} is proved.
\end{proof}

\begin{corollary}\label{cor:gen1} 
For $A=H_\ast$, the algebra $\COHA^T_{\Higgs(X)}\otimes \Q(t)$ is generated over $\Q(t)$ by the collection of subspaces $H_\ast(\bLam_{(\alpha)})$ for $\alpha \in (\Z^2)^+$.
\end{corollary}

\begin{corollary}\label{cor:gen2} 
For $A=H_\ast$, $\COHA_{\bLam}$ is generated as an $\HH$-module algebra by the collection of elements $[\bLam_{(\alpha)}]$ for $\alpha \in (\Z^2)^+$.
\end{corollary}
\begin{proof} 
It suffices to observe that, by Poincar\'e duality for the stack $\bCoha$,  $\HH \cdot [\bLam_{(\alpha)}]=H_\ast(\bLam_{(\alpha)})$.
\end{proof}

\end{document}